\documentclass[a4paper,11pt]{article}
\usepackage{authblk}
\title{Higher genus theory}

\usepackage{amsmath, amssymb, mathrsfs, amsthm, enumitem}
\usepackage{hyperref}
\usepackage[all]{xy}
\usepackage[top=2.5cm, bottom=2.5cm]{geometry}

\newcommand{\Q}{\mathbb{Q}}

\newcommand{\bmt}{\begin{pmatrix}}
\newcommand{\emt}{\end{pmatrix}}
\newcommand{\bsm}{\left(\begin{smallmatrix}}
\newcommand{\esm}{\end{smallmatrix}\right)}

\theoremstyle{definition}
\newtheorem{definition}{Definition}[section]

\newtheorem{remark}[definition]{Remark}

\theoremstyle{plain}
\newtheorem{proposition}[definition]{Proposition}
\newtheorem{lemma}[definition]{Lemma}
\newtheorem{theorem}[definition]{Theorem}

\theoremstyle{remark}

\numberwithin{equation}{section}

\renewcommand{\phi}{\varphi}

\newcounter{nootje}
\newcommand{\beq}{\begin{equation}}
\newcommand{\eeq}{\end{equation}}
\newcommand{\beqs}{\begin{equation*}}
\newcommand{\eeqs}{\end{equation*}}

\title{\vspace{-\baselineskip}\sffamily\bfseries Higher Genus Theory}
\author[1]{Peter Koymans\thanks{Vivatsgasse 7, 53111  Bonn, Germany, koymans@mpim-bonn.mpg.de}}
\author[1]{Carlo Pagano\thanks{Vivatsgasse 7, 53111  Bonn, Germany, carlein90@gmail.com}}
\affil[1]{Max Planck Institute for Mathematics, Bonn}
\date{\today}

\begin{document}
\maketitle
	
\begin{abstract}
In $1801$, Gauss found an explicit description, in the language of binary quadratic forms, for the $2$-torsion of the narrow class group and dual narrow class group of a quadratic number field. This is now known as Gauss's genus theory. In this paper we extend Gauss's work to the setting of multi-quadratic number fields. To this end, we introduce and parametrize the categories of expansion groups and expansion Lie algebras, giving an explicit description for the universal objects of these categories. This description is inspired by the ideas of Smith \cite{smith2} in his recent breakthrough on Goldfeld's conjecture and the Cohen--Lenstra conjectures.

Our main result shows that the maximal unramified multi-quadratic extension $L$ of a multi-quadratic number field $K$ can be reconstructed from the set of generalized governing expansions supported in the set of primes that ramify in $K$. This provides a recursive description for the group $\text{Gal}(L/\mathbb{Q})$ and a systematic procedure to construct the field $L$. A special case of our main result gives an upper bound for the size of $\text{Cl}^{+}(K)[2]$.
\end{abstract}

\section{Introduction}
The narrow class group of a number field $K$, which we denote by $\text{Cl}^{+}(K)$, is one of the most fundamental and, yet mysterious, objects in arithmetic. Its study, initiated by Gauss \cite{Gauss} in the language of binary quadratic forms, has triggered a substantial part of the developments of algebraic number theory since $1801$. In his dissertation Gauss reported what still is one of the very few explicit results about the class group. Namely, for a quadratic number field $K$, Gauss provided an explicit description for the $2$-torsion of $\text{Cl}^{+}(K)$ and of $\text{Cl}^{+}(K)^{\vee}$, the dual of the narrow class group. This description is given in terms of the primes dividing the discriminant $\Delta_{K/\mathbb{Q}}$. In particular from such a description Gauss was able to conclude that
$$
\text{dim}_{\mathbb{F}_2} \text{Cl}^{+}(K)[2] = \omega(\Delta_{K/\mathbb{Q}})-1
$$
for a quadratic number field $K$. Here $\omega(m)$ denotes the number of distinct prime factors of an integer $m$. It is not difficult to generalize this work to the case where $K$ is a cyclic prime degree extension of $\mathbb{Q}$. For that case, Gauss's work provides an explicit description for the $\text{Gal}(K/\mathbb{Q})$-invariants of $\text{Cl}^{+}(K)$ and $\text{Cl}^{+}(K)^{\vee}$. This is now known as Gauss's \emph{genus theory}. For a historical overview and a function field version, see respectively the work of Lemmermeyer \cite{Lemmermeyer} and Cornelissen \cite{Cornelissen}.

When no such explicit description is available, subsequent research has attempted to show that the class group behaves as ``randomly" as possible as $K$ varies in the family of quadratic number fields (or more general families). A precise notion of randomness was proposed by Cohen and Lenstra \cite{cohen--lenstra} in their conjectures for the class group. These conjectures were later refined by Gerth \cite{gerth} in order to include the case of the $2$-Sylow of the class group. Here the issue was precisely to isolate the ``random part" from the ``explicit part", that is $\text{Cl}^{+}(K)[2]$, where Gauss's genus theory applies. 

Spectacular progress on these conjectures has recently been made by Smith \cite{smith2}, where he proved Gerth's extension of the Cohen--Lenstra conjectures for $\text{Cl}^{+}(K)[2^{\infty}]$. A crucial ingredient of Smith's work is the notion of a \emph{governing expansion}. These are rather explicit objects that naturally provide elements of $\text{Cl}^{+}(K)^{\vee}[2]$ when $K$ is a \emph{multi-quadratic} number field. This brings us to the main topic of this paper.\\

$\mathbf{Questions:}$ Let $K$ be a multi-quadratic number field: \\
$(1)$ Is there a description of $\text{Cl}^{+}(K)^{\vee}[2]$? \\
$(2)$ Do Smith's governing expansions provide a set of generators for such a space? \\
$(3)$ How large can such a space be in terms of the primes ramifying in $K$? \\

The present work provides an affirmative answer to the first two questions and answers the third question by means of an upper bound. We call the ensemble of such results \emph{higher genus theory}, a terminology that will be explained later in this introduction. We shall begin by explaining our first three main results, which address the third question. They form the crudest manifestation of higher genus theory. Recall that for a positive integer $m$ we denote by $\omega(m)$ the number of distinct prime factors of $m$. 

\begin{theorem} 
\label{theoremA}
Let $n$ be a positive integer and let $a_1, \dots, a_n$ be square-free numbers in $\mathbb{Z}_{\geq 2}$ that are pairwise coprime and that have only prime factors congruent to $1$ modulo $4$. Then
$$
\emph{dim}_{\mathbb{F}_2} \emph{Cl}^{+}(\mathbb{Q}(\sqrt{a_1}, \dots, \sqrt{a_n}))[2] \leq \omega(a_1 \cdot \ldots \cdot a_n) \cdot 2^{n - 1} - 2^n + 1. 
$$
\end{theorem}

To the best of our knowledge, this is the first non-trivial upper bound for the group $\text{Cl}^{+}(\mathbb{Q}(\sqrt{a_1}, \dots, \sqrt{a_n}))[2]$ appearing in the literature. A trivial upper bound is $\ll_\epsilon \Delta_{K/\Q}^{1/2 + \epsilon}$ by the Brauer--Siegel theorem. If $n = 1$, we see that our upper bound is in concordance with Gauss's genus theory. In contrast, for $n \geq 2$ one can show that merely prescribing the values of $\omega(a_1), \dots, \omega(a_n)$ does not force $\text{dim}_{\mathbb{F}_2} \text{Cl}^{+}(\mathbb{Q}(\sqrt{a_1}, \dots, \sqrt{a_n}))[2]$ to attain a unique value. We hope to show in future work that the upper bound in Theorem \ref{theoremA} is actually sharp in a wide number of cases. 

We say that a vector $(a_1, \dots, a_n)$ is \emph{acceptable} if $a_1, \dots, a_n$ are square-free numbers in $\mathbb{Z}_{\geq 2}$ that are pairwise coprime and have only prime factors congruent to $1$ modulo $4$. This condition can be partly relaxed in all our theorems. It is possible to remove the restriction on the prime factors of $a_i$ with only minor modifications and some case distinctions, but the coprimality condition among the $a_i$ may be more difficult to remove.

We say that an acceptable vector is \emph{maximal} if the bound of Theorem \ref{theoremA} is attained. Our next result provides the following neat recursive characterization of maximal vectors. For a number field $K$, we denote by $H_2^{+}(K)$ the maximal elementary exponent $2$ extension of $K$ that is unramified at all finite places; this is the field corresponding to $\text{Cl}^{+}(K)^{\vee}[2]$ by class field theory. For a positive integer $m$ we denote by $[m]$ the set of positive integers no larger than $m$. 
 
\begin{theorem}
\label{theoremC}
Let $n$ be a positive integer and let $(a_1, \dots, a_n)$ be an acceptable vector. Then the following are equivalent. \\
$(a)$ The vector $(a_1, \dots, a_n)$ is maximal, i.e. 
$$
\emph{dim}_{\mathbb{F}_2} \emph{Cl}^{+}(\mathbb{Q}(\sqrt{a_1}, \dots, \sqrt{a_n}))[2] = \omega(a_1 \cdot \ldots \cdot a_n) \cdot 2^{n - 1}-2^n+1.
$$
$(b)$ For every $j \in [n]$, the vector $(a_h)_{h \neq j}$ is maximal and every prime divisor $p$ of $a_j$ splits completely in $H_2^{+}(\mathbb{Q}(\{\sqrt{a_m}\}_{m \in [n]-\{j\}}))$. \\
$(c)$ For every $j \in [n]$, the vector $(a_h)_{h \neq j}$ is maximal and for every prime divisor $p$ of $a_j$, one (or equivalently any) prime above $p$ in the field $\mathbb{Q}(\{\sqrt{a_m}\}_{m \in [n]-\{j\}})$ belongs to  $2\emph{Cl}^{+}(\mathbb{Q}(\{\sqrt{a_m}\}_{m \in [n]-\{j\}}))$.  
\end{theorem}

We remark that condition $(b)$ is emptily satisfied when $n=1$, so that Theorem \ref{theoremA} and Theorem \ref{theoremC} together recover the usual Gauss's genus theory. We also observe that the equivalence between conditions $(b)$ and $(c)$ in Theorem \ref{theoremC} follows directly from class field theory. Hence the non-trivial assertion is the equivalence between condition $(a)$ and $(b)$ (or equivalently between $(a)$ and $(c)$). 

Fr\"ohlich \cite{frolich} systematically investigated a certain subspace of $\text{Cl}^{+}(K)^{\vee}[2]$ for multi-quadratic fields. To explain which subspace, we recall that elements of $\text{Cl}^{+}(K)^{\vee}[2]$ have a natural notion of complexity, leading to a filtration
$$
\{0\} = \text{Gn}(K,0) \subseteq \text{Gn}(K,1) \subseteq \text{Gn}(K, 2) \subseteq \dots \subseteq \text{Gn}(K, j) \subseteq \cdots,
$$
exhausting the space $\text{Cl}^{+}(K)^{\vee}[2]$. Intuitively, the index $j$ in the filtration measures the extent to which the corresponding Galois groups over $\mathbb{Q}$ are non-commutative. The precise definition is as follows. For $j \in \mathbb{Z}_{\geq 1}$, a character $\chi \in \text{Cl}^{+}(K)^{\vee}[2]$ belongs to the $j$-th term of the filtration if it vanishes on all $(j+1)$-th nested commutators\footnote{Observe that this is well defined since $K/\mathbb{Q}$ is abelian and $j \geq 1$.} with entries in $G_{\mathbb{Q}}$. Here we view the character $\chi$ as a homomorphism $G_K \rightarrow \mathbb{F}_2$ by class field theory, where $G_K$ denotes the absolute Galois group of a number field $K$. If an element is in the $j$-th but not in the $(j - 1)$-th term of the filtration, then we say that it has nilpotency $j$.

The goal of higher genus theory is to describe all the spaces $\text{Gn}(K, j)$. Gauss's genus theory handles the case $j=1$. The work of Fr\"ohlich extends this to $j=2$. The present paper handles all $j \in \mathbb{Z}_{\geq 1}$. 

The jump from nilpotency $j \in \{1, 2\}$ to nilpotency $j \in \mathbb{Z}_{\geq 1}$ has a parallel in the recent dramatic developments on the Cohen--Lenstra conjectures: this is no coincidence. Using the work of R\'edei \cite{Redei}, it is possible to prove that the $4$-rank and the $8$-rank of class groups of quadratic number fields follow the Cohen--Lenstra conjectures. For the former, see the work of Fouvry--Kl\"uners \cite{FK1, fouvry--kluners} and for the latter see Smith's work \cite{smith1} under GRH predating his major breakthrough \cite{smith2}. There has also been recent interest in non-abelian generalizations of the Cohen--Lenstra conjectures studied from a heuristical standpoint by \cite{BBH, BW, LWZ, WW} and from a statistical viewpoint by \cite{klys}. 

For the $4$-rank, Gauss's genus theory is enough to obtain R\'edei's criterion employed by Fouvry and Kl\"uners. For the $8$-rank, R\'edei found a criterion in terms of certain extensions of nilpotency class $2$, which we shall label \emph{R\'edei fields}; for a modern exposition on R\'edei fields and their connection with the $8$-rank, see the work of Stevenhagen \cite{stevenhagen}. On the one hand, R\'edei fields are precisely the fields that Fr\"ohlich uses to analyze $\text{Gn}(K,1)$. On the other hand Smith's notion of governing expansions provides a generalization of R\'edei fields for any nilpotency class. As we shall now explain, the present work completes this picture and shows that these fields do indeed provide a complete description for $\text{Gn}(K, j)$ for all $j$. 

From now on we shall use the notation $H_2^{+}(a_1, \dots, a_n):=H_2^{+}(\mathbb{Q}(\sqrt{a_1}, \dots, \sqrt{a_n}))$ for an acceptable vector $(a_1, \dots, a_n)$. Let $(a_1, \dots, a_n)$ be an acceptable vector and denote $k_i:=\omega(a_i)$. We abstract the most fundamental features of $\text{Gal}(H_2^{+}(a_1, \dots, a_n)/\mathbb{Q})$ in the notion of a $[(k_1, \dots, k_n)]$\emph{-expansion group} in Section \ref{Expansion algebraic}. This is an algebraic structure, consisting of a group with certain extra data. 

Our crucial step is to show that the Lie algebra attached, by means of the descending central series, to a $[(k_1, \dots, k_n)]$-expansion group is a highly constrained one: this leads us to the notion of a $[(k_1, \dots, k_n)]$\emph{-expansion Lie algebra}. We use these constraints to show that the dimension of a $[(k_1, \dots, k_n)]$-expansion Lie algebra is always bounded by 
$$
(k_1+ \dots + k_n) \cdot 2^{n - 1} - 2^n + 1 + n.
$$
This is obtained by bounding the dimension of certain tensor spaces encoding all the constraints shared by a $[(k_1, \dots, k_n)]$-expansion Lie algebra. This calculation is done in Section \ref{Auxiliary tensor spaces}. Owing to this step we deduce the same bound for a $[(k_1, \dots, k_n)]$-expansion group. Already at this stage we are able to establish the following inequality.

\begin{theorem} 
\label{controlling gn} 
Let $(a_1, \dots, a_n)$ be an acceptable vector. Then
$$
\emph{dim}_{\mathbb{F}_2} \frac{\emph{Gn}(\mathbb{Q}(\sqrt{a_1}, \dots, \sqrt{a_n}), j)}{\emph{Gn}(\mathbb{Q}(\sqrt{a_1}, \dots, \sqrt{a_n}), j - 1)} \leq \omega(a_1 \cdot \ldots \cdot a_n) \cdot \binom{n - 1}{j - 1} - \binom{n}{j},
$$
for all $j \in \mathbb{Z}_{\geq 1}$.
\end{theorem} 

Theorem \ref{controlling gn} sharpens the conclusion of Theorem \ref{theoremA}. To upgrade these inequalities to a full description of all $[(k_1, \dots, k_n)]$-expansion groups, so in particular of the Galois group $\text{Gal}(H_2^{+}(a_1, \dots, a_n)/\mathbb{Q})$, we proceed as follows. We use an abstracted version of Smith's governing expansions to construct a $[(k_1, \dots, k_n)]$-expansion group and a $[(k_1, \dots, k_n)]$-expansion Lie algebra of the maximal possible size $2^m$ where $m$ equals
$$
(k_1+ \dots + k_n) \cdot 2^{n - 1} - 2^n + 1 + n. 
$$
Part of this data is a pair $(\widetilde{\mathcal{G}}([(k_1, \dots, k_n)]),(g_1, \dots, g_{k_1+ \dots + k_n}))$, where $\widetilde{\mathcal{G}}([(k_1, \dots, k_n)])$ is a certain finite $2$-group and $\{g_1, \dots, g_{k_1+ \dots + k_n}\}$ is a set of generating involutions of  $\widetilde{\mathcal{G}}([(k_1, \dots, k_n)])$. A similar construction is carried out in the case of Lie algebras. We shall refer to these two structures as the \emph{governing group} and \emph{governing algebra}.

In parallel, we first show that in these two categories the set of morphisms between two objects always has at most $1$ element, and a morphism is always a surjective group or Lie algebra homomorphism. Second we show that there exists a \emph{universal} object, namely one that maps (uniquely and surjectively) to any other of them: this fact is proved in a \emph{soft} manner in two different ways. In particular these soft arguments give no clue on the shape of this universal object and no a priori control on its size, apart from the above mentioned upper bound. 

But the governing group and the governing algebra reach precisely that upper bound. Therefore we obtain the non-trivial conclusion that the governing group and the governing Lie algebra must be the universal objects among $[(k_1, \dots, k_n)]$-expansion groups and Lie algebras respectively. Altogether, this culminates in the following considerable refinement of Theorem \ref{theoremA}. 

\begin{theorem} 
\label{Presenting the group}
Let $(k_1, \dots, k_n)$ be in $\mathbb{Z}_{\geq 1}^{n}$. Let $(a_1, \dots, a_n)$ be an acceptable vector with $\omega(a_i) = k_i$ for every $i \in [n]$. List the prime factors of $a_1 \cdot \ldots \cdot a_n$ as $\{p_1,  \dots, p_{k_1+ \dots + k_n}\}$ in such a way that the prime factors of $a_i$ are $\{p_{1 + \sum_{1 \leq j \leq i - 1}k_j}, \dots, p_{\sum_{1 \leq j \leq i} k_j}\}$ and for each $p_i$ make a choice $\sigma_{i}$ of an inertia element in $\emph{Gal}(H_2^{+}(a_1, \dots, a_n)/\mathbb{Q})$. 

Then the assignment $g_i \mapsto \sigma_i$ extends uniquely to a group epimorphism
$$
\phi: \widetilde{\mathcal{G}}([(k_1, \dots, k_n)]) \twoheadrightarrow \emph{Gal}(H_2^{+}(a_1, \dots, a_n)/\mathbb{Q}). 
$$
Furthermore, the map $\phi$ is an isomorphism if and only if
$$
\emph{dim}_{\mathbb{F}_2} \emph{Cl}^{+}(\mathbb{Q}(\sqrt{a_1}, \dots, \sqrt{a_n}))[2] = \omega(a_1 \cdot \ldots \cdot a_n) \cdot 2^{n - 1} - 2^n + 1.
$$
\end{theorem}

Our final result provides a \emph{recursive description} of each of the spaces 
$$
\text{Gn}(\mathbb{Q}(\sqrt{a_1}, \dots, \sqrt{a_n}),j),
$$
for each $(a_1, \dots ,a_n)$ and $j \in \mathbb{Z}_{\geq 1}$. In particular it gives a substantial generalization of Theorem \ref{theoremC}. Let $(a_1, \dots, a_n)$ be an acceptable vector. For each $T \subseteq [n]$ denote by $K_T$ the field $\mathbb{Q}((\sqrt{a_h})_{h \in T})$. Let $j$ be in $\mathbb{Z}_{\geq 1}$. Then the material of Section \ref{Reconstructing from corners} gives, for each $T \subseteq [n]$ and $j \in [|T|]$, a certain finite dimensional vector space
$$
\Phi_j(a_{h})_{h \in T} \subseteq \text{Cont-Map}(G_{\mathbb{Q}},\mathbb{F}_2),
$$
where $\text{Cont-Map}(G_{\mathbb{Q}},\mathbb{F}_2)$ denotes the space of continuous $1$-cochains from $G_{\mathbb{Q}}$ to $\mathbb{F}_2$.\footnote{These spaces will be reinterpreted, via Shapiro's Lemma, as certain $1$-cocycles in $\mathbb{F}_2[\mathbb{F}_2^{n}]$.} The space $\Phi_j(a_{h})_{h \in T}$ has the property that the restriction to $G_{K_T}$ yields a surjective homomorphism
$$
\Phi_j(a_{h})_{h \in T} \twoheadrightarrow \text{Gn}(K_T,j).
$$
Furthermore, the spaces $\Phi_j(a_h)_{h \in T}$, as $j$ and $T$ vary, are linked together with the following additional data. Namely, whenever $j \geq 2$ and $i \in T$, it turns out that one has a natural map
$$
P_i:\Phi_j(a_h)_{h \in T} \to \Phi_{j-1}(a_h)_{h \in T-\{i\}},
$$
which, after applying restriction to $G_{K_T}$ and $G_{K_{T-\{i\}}}$ respectively, becomes the natural norm map between character groups. The operators $P_i$ commute and this allows us to define an operator $P_S$ for each $S \subseteq [n]$. These operators reduce the complexity of a map in two ways: they lower the nilpotency degree and the degree of the multi-quadratic fields. However, despite each of the operators $P_S$ individually \emph{simplifies} a map $\Phi$, when considered altogether they encode the map $\Phi$ with the following universal recipe. Namely one has the key equation
\begin{align}
\label{Univ.exp eq.}
(d\Phi)(\sigma,\tau)=\sum_{\emptyset \neq S \subseteq [n]}\chi_{S}(\sigma)P_S(\Phi)(\tau).
\end{align}
Here $\chi_{S}=\prod_{i \in S}\chi_{a_i}$, where the product is done in $\mathbb{F}_2$.\footnote{Here $\chi_{a}$ denotes the quadratic character corresponding to the extension $\mathbb{Q}(\sqrt{a})/\mathbb{Q}$.} Observe that equation (\ref{Univ.exp eq.}) determines the coset of $\Phi$ with respect to the space spanned by the set of characters $\{\chi_p\}_{p|a_1 \cdot \ldots \cdot a_n}$. 

Equation (\ref{Univ.exp eq.}) is a \emph{universal version} of Smith's governing expansion equation (see \cite[eq. (2.2)]{smith2}), which is a special case. Incidentally, we provide an alternative way to think about the notion of a governing expansion. Namely we show that to give a governing expansion is tantamount to giving an epimorphism from $G_{\mathbb{Q}}$ to the group
$$
\mathbb{F}_2[\mathbb{F}_2^m] \rtimes \mathbb{F}_2^m.
$$
This dictionary is established in Section \ref{expansion maps as coordinates of monomials}, which provides the direct link between governing expansions and expansion groups. 

Our final theorem provides an inverse to equation (\ref{Univ.exp eq.}), which allows us to construct the space $\Phi_{j}(a_1, \dots, a_n)$ out of the (lower complexity) spaces $(\Phi_{j-1}((a_h)_{h \neq i}))_{i \in [n]}$, for each $j \in \mathbb{Z}_{\geq 2}$. We define $\text{Comm-Vect}_{j}(a_1, \dots, a_n)$ to be the set of vectors $(\Phi_1, \dots, \Phi_n)$ with $\Phi_i \in \Phi_{j-1}((a_h)_{h \neq i})$ and with the property that 
$$
P_i(\Phi_k) = P_k(\Phi_i),
$$ 
for each distinct $i, k \in [n]$. Then one writes down
$$
\theta((\Phi_i)_{i \in [n]}) := \sum_{i \in [n]} \chi_{\{i\}}(\sigma)\Phi_i(\tau) + \sum_{\substack{B \subseteq [n] \\ \#B \geq 2}} \chi_{B}(\sigma) P_B((\Phi_i)_{1 \leq i \leq n})(\tau),
$$
where we can unambiguously define $P_B((\Phi_i)_{1 \leq i \leq n})$ as the composition of the maps $P_i$ as $i$ varies in $B$. From equation (\ref{Univ.exp eq.}) it follows that $\theta((\Phi_i)_{i \in [n]})$ is a $2$-\emph{cocycle}. Denote by $\text{Comm-Vect}_{j}^{\circ}(a_1, \dots, a_n)$ the subspace of $\text{Comm-Vect}_{j}(a_1, \dots, a_n)$ yielding trivial classes in $H^2(G_{\mathbb{Q}},\mathbb{F}_2)$. We have the following. 

\begin{theorem} 
\label{Last thm}
For every $(\Phi_1, \dots, \Phi_n)$ in $\emph{Comm-Vect}_{j}^{\circ}(a_1, \dots, a_n)$ there exists $\Phi$ inside $\Phi_j(a_1, \dots, a_n)$ such that 
$$d\Phi=\theta((\Phi_i)_{i \in [n]}).
$$ 
Furthermore such a $\Phi$ must satisfy also 
$$(P_1(\Phi), \dots, P_n(\Phi))=(\Phi_1, \dots, \Phi_n).
$$
Conversely, for each $\Phi \in \Phi_j(a_1, \dots, a_n)$, we have that 
$$d\Phi=\theta(P_i(\Phi)_{i \in [n]}),
$$ 
and so $(P_1(\Phi), \dots, P_n(\Phi)) \in \emph{Comm-Vect}_{j}^{\circ}(a_1, \dots, a_n)$. 
\end{theorem}

Theorem \ref{Last thm} gives a procedure to access each of the spaces $\text{Gn}(K_{[n]}, j)$ inductively. Indeed, starting from the spaces $\Phi_{j-1}((a_h)_{h \in [n]-\{i\}})$, for each $i$ in $[n]$, one determines which $\underline{v} \in \text{Comm-Vect}_{j}(a_1, \dots, a_n)$ yield an unobstructed class $\theta(\underline{v}) \in H^2(G_{\mathbb{Q}},\mathbb{F}_2)$. This gives the space $\text{Comm-Vect}_{j}^{\circ}(a_1, \dots, a_n)$. Then Theorem \ref{Last thm} guarantees that each of the embedding problems in $\text{Comm-Vect}_{j}^{\circ}(a_1, \dots, a_n)$ admits as a solution an $1$-cochain that yields an element of $\Phi_j(a_1, \dots, a_n)$. Furthermore Theorem \ref{Last thm} tells us that all elements of $\Phi_j(a_1, \dots, a_n)$ arise in this way. Finally, restricting $\Phi_j(a_1, \dots, a_n)$ to $G_{K_{[n]}}$ yields $\text{Gn}(K_{[n]},j)$. 

The space $\Phi_1(a_1, \dots, a_n)$ simply consists of the span of the functions $\{\chi_{A}\}_{\emptyset \neq A \subseteq [n]}$ and all characters $\{\chi_p\}_{p \mid a_1 \cdot \ldots \cdot a_n}$. Hence this procedure gives a recursive description of $\text{Gn}(K_{[n]}, j)$ purely in terms of the arithmetic of the ground field $\mathbb{Q}$.

\section*{Acknowledgments}
This work owes an evident intellectual debt to Alexander Smith's work \cite{smith2}. We would also like to thank Alexander for several clarifying emails and conversations about his work. Especially for the one where he shared with us his insight that something along the line of the present work might be possible.

We are very grateful to Jared Asuncion for running several computer experiments that made visible the pattern of Theorem \ref{theoremA}. This greatly motivated us to find a proof and to develop a genus theory for every nilpotency class.    

We wish to thank Hendrik Lenstra for an insightful conversation predating this work, where he mentioned the group $\mathbb{F}_2[\mathbb{F}_2^n] \rtimes \mathbb{F}_2^n$ as an example of a group of exponent $4$ and large nilpotency class. These objects turned out to provide precisely the adequate language to understand governing expansions and to construct universal expansion groups.   

We thank Ren\'e Schoof for an e-mail on the case $j = 2$ of Theorem \ref{controlling gn}.

We thank Adam Morgan for suggesting a reformulation of the spaces $\Phi_j(a_{h})_{h \in T}$ in terms of $1$-cocycles in the group ring. This has simplified the proof of Theorem \ref{Last thm}.

The authors wish to thank the Max Planck Institute for Mathematics in Bonn for its financial support, great work conditions and an inspiring atmosphere. 

\section{Auxiliary tensor spaces} \label{Auxiliary tensor spaces}
In this section we define certain tensor spaces and carry out a crucial combinatorial calculation contained in Proposition \ref{Consistent tensors are governing tensors} and Proposition \ref{Consistent tensors are governing tensors more general}. This calculation forms the foundation of the algebraic material of Section \ref{Expansion algebraic} and therefore of our main arithmetical applications, which are given in Section \ref{Expansion arithmetic}. 

For any positive integer $m$ we recall that $[m]$ denotes the set of positive integers that are no bigger than $m$. For a set $S$, we denote by $V_{S}$ the vector space $\mathbb{F}_2^{(S)}$ of formal $\mathbb{F}_2$-linear combinations of the elements of $S$. In this way each element $s \in S$ gives a vector $e_s \in V_{S}$ and the collection $\{e_s\}_{s \in S}$ gives a basis for $V_{S}$. We denote by $\{\chi_s\}_{s \in S}$ the unique linear functional from $V_S$ to $\mathbb{F}_2$ defined by the equation $\chi_s(e_{s'})=\delta_{s,s'}$ for each $s, s' \in S$. Let $n, i \in \mathbb{Z}_{\geq 0}$ and define 
$$
\text{Multi}(V_{[n]}, i) := \{b : V_{[n]}^i \to \mathbb{F}_2 \ | \ b \text{ multi-linear}\}.
$$
Observe that a basis for $\text{Multi}(V_{[n]}, i)$ is $\{\chi_{h_1} \otimes \dots \otimes \chi_{h_i}\}_{(h_1, \dots,h_i) \in [n]^i}$, where $\chi_{h_1} \otimes \dots \otimes \chi_{h_i}$ is short for the map $(\sigma_1, \dots, \sigma_i) \mapsto \chi_{h_1}(\sigma_1) \cdot \ldots \cdot \chi_{h_i}(\sigma_i)$. Here the product on the right hand side is the usual product in $\mathbb{F}_2$. If $i \geq 2$, then we put for every $A \subseteq [n]$ with $\#A = i$ and $x \in A$
\begin{align}
\label{eGovTen}
\phi_{(A, x)} := \sum_{\substack{\tau \in \text{Isom}_{\text{Set}}([i], A) \\ \tau(i - 1)=x \ \text{or} \ \tau(i)=x}} \chi_{\tau(1)} \otimes \dots \otimes \chi_{\tau(i)} \in \text{Multi}(V_{[n]}, i).
\end{align}
If $i = 1$ and $A = \{x\}$, then we put $\phi_{(A, x)} := \chi_x$. We shall refer to such tensors as \emph{governing tensors}. We denote by 
$$
\text{Gov}(V_{[n]}, i) := \langle \{\phi_{(A, x)}\}_{A \subseteq [n], \#A = i, x \in A} \rangle.
$$
We have the following elementary fact.

\begin{proposition} 
\label{counting}
Let $i$ be in $\mathbb{Z}_{\geq 2}$. Then we have that
$$
\emph{dim}_{\mathbb{F}_2} \emph{Gov}(V_{[n]}, i) = (i - 1) \cdot \binom{n}{i}.
$$
\end{proposition}

\begin{proof}
We can assume that $i \leq n$, otherwise both sides of the equality are $0$. For each $A \subseteq [n]$ with $\#A = i$ we deduce from equation (\ref{eGovTen}) the identity
$$
\sum_{x \in A} \phi_{(A, x)} = 0.
$$
From this it follows that 
$$
\text{dim}_{\mathbb{F}_2} \text{Gov}(V_{[n]}, i) \leq (i - 1) \cdot \binom{n}{i}.
$$
Since $\{\chi_{h_1} \otimes \dots \otimes \chi_{h_i}\}_{(h_1, \dots,h_i) \in [n]^i}$ is a basis for $\text{Multi}(V_{[n]}, i)$, it is transparent that there are no further relations among the elements of the set $\{\phi_{A,x}\}_{A \subseteq [n], \#A = i, x \in A}$. This gives the desired equality. 
\end{proof}

For $i \in \mathbb{Z}_{\geq 0}$, we introduce the subspace $\widetilde{\text{Multi}}(V_{[n]}, i)$ of $\text{Multi}(V_{[n]}, i)$, defined as
$$
\widetilde{\text{Multi}}(V_{[n]}, i):=\langle \{\chi_{\tau(1)} \otimes \dots \otimes \chi_{\tau(i)} \}_{\tau \in \text{Inj}([i], [n])} \rangle,
$$
where $\text{Inj}([i], [n])$ denotes the set of injective maps from $[i]$ to $[n]$. We can alternatively characterize $\widetilde{\text{Multi}}(V_{[n]}, i)$ as the space of multi-linear maps from $V^i$ to $\mathbb{F}_2$ that vanish if we put the same element $e_h$ in two distinct coordinates for some $h \in [n]$. Recall that for any subset $B$ of $[n]$, the symbol $V_B$ denotes the subspace of $V_{[n]}$ spanned by $\{e_j\}_{j \in B}$ and we put
$$\widetilde{\text{Multi}}(V_B, i):=\langle \{\chi_{\tau(1)} \otimes \dots \otimes \chi_{\tau(i)} \}_{\tau \in \text{Inj}([i], B)} \rangle.
$$
Observe that if $i \in \mathbb{Z}_{\geq 1}$ and $b \in \widetilde{\text{Multi}}(V_B, i)$, then $b$ can be reconstructed from the $\# B$ maps $(b(e_j, -))_{j \in B}$ with each $b(e_j, -) \in  \widetilde{\text{Multi}}(V_{B - \{j\}}, i - 1)$. Conversely, given any such collection of maps $(b_j)_{j \in B} \in \prod_{j \in B} \widetilde{\text{Multi}}(V_{B - \{j\}}, i - 1)$, there exists a unique $b$ in $\widetilde{\text{Multi}}(V_B, i)$ with $b(e_j, -)=b_j(-)$ for each $j \in B$. In other words we have a natural identification
$$
P(V_B, i) : \widetilde{\text{Multi}}(V_B, i) \to \prod_{j \in B} \widetilde{\text{Multi}}(V_{B - \{j\}}, i - 1),
$$
i.e. the map $P(V_B, i)$ is an isomorphism of vector spaces. The following calculation will often be helpful.

\begin{proposition} 
\label{P of governing tensors}
Let $i$ be in $\mathbb{Z}_{\geq 3}$ and $A \subseteq B \subseteq [n]$ with $\#A = i$. Let $x \in A$. Then 
$$
P(V_B, i)(\phi_{(A, x)}) = (0)_{j \not \in A} \times (\phi_{(A - \{j\}, x)})_{j \in A - \{x\}} \times (0)_{j = x}.
$$
\end{proposition}

\begin{proof}
This follows directly from the definition and equation (\ref{eGovTen}).
\end{proof}

Let $B \subseteq [n]$ and $i \in \mathbb{Z}_{\geq 1}$. We are going to define a subspace
$$
\text{Cons}(V_B, i) \subseteq \widetilde{\text{Multi}}(V_B, i)
$$ 
that will be used in the arithmetical Section \ref{Expansion arithmetic} by means of the intermediate notion of $n$-expansion explored in Section \ref{Expansion algebraic}. This will handle the case of multi-quadratic fields obtained from quadratic fields with prime discriminant. Since its definition treats asymmetrically the cases $i = 2$ and $i = 3$ with respect to the other cases we will reserve special notation for these cases. The nomenclature used in these cases shall become clear in Section \ref{Expansion algebraic}. For $i = 1$ we put $\text{Cons}(V_B, 1) := \text{Gov}(V_B, 1)$ and for $i = 2$ we put
$$
\text{Cons}(V_B, 2) := \text{Sym}(V_B, 2),
$$
where $\text{Sym}(V_B, 2)$ is defined to be the subspace of $b$ in $\widetilde{\text{Multi}}(V_B, 2)$ such that 
$$
b(\sigma_1, \sigma_2)+b(\sigma_2, \sigma_1)=0 \ \emph{(Symmetry)}.
$$ 
For $i = 3$ we put
$$
\text{Cons}(V_B, 3):=\text{Hall--Witt}(V_B, 3),
$$
where $\text{Hall--Witt}(V_B, 3)$ is defined to be the subspace of $\widetilde{\text{Multi}}(V_B, 3)$ consisting of those $b$ such that $P(V_B, 3)(b) \in \prod_{j \in B} \text{Sym}(V_{B - \{j\}}, 2)$ and
$$
b(\sigma_1, \sigma_2, \sigma_3) + b(\sigma_3, \sigma_1, \sigma_2) + b(\sigma_2, \sigma_3, \sigma_1) = 0 \ (\emph{Hall--Witt equation})
$$
for all $\sigma_1, \sigma_2, \sigma_3 \in \{e_j\}_{j \in [n]}$. We make the important observation that each set of three distinct elements $\sigma_1, \sigma_2, \sigma_3 \in \{e_j\}_{j \in [n]}$ gives a unique Hall--Witt equation thanks to the requirement that $P(V_B, 3)(b) \in \prod_{j \in B}\text{Sym}(V_{B - \{j\}}, 2)$; as one permutes $\sigma_1, \sigma_2, \sigma_3$ and uses that $P(V_B, 3)(b) \in \prod_{j \in B}\text{Sym}(V_{B - \{j\}}, 2)$, one gets literally the same equation. 

Finally let $i \geq 4$. Then we put $\text{Cons}(V_B, i)$ to be the subspace of $\widetilde{\text{Multi}}(V_B, i)$ consisting of those $b$ such that $P(V_B, i)(b) \in \prod_{j \in B}\text{Cons}(V_{B - \{j\}}, i - 1)$ and
\begin{align*}
\pi_{j_1} \circ &P(V_{B-\{j_2\}}, i - 1) \circ \pi_{j_2} \circ P(V_B, i)(b) = \\
&\pi_{j_2} \circ P(V_{B-\{j_1\}}, i - 1) \circ \pi_{j_1} \circ P(V_B, i)(b) \ (\emph{Commutativity})
\end{align*}
for all distinct $j_1, j_2 \in B$. Here $\pi_j$ denotes the projection on the $j$-th component. 

\begin{remark} 
\label{asymmetry 3 and 4}
For $i = 2$ the Commutativity equation 
$$
\pi_{j_1} \circ P(V_{B-\{j_2\}}, 1)\circ \pi_{j_2}\circ P(V_B, 2)(b)=\pi_{j_2} \circ P(V_{B - \{j_1\}}, 1) \circ \pi_{j_1} \circ P(V_B, 2)(b)
$$
is equivalent to the Symmetry equation. For $i = 3$ the Commutativity equation is however, in general, \emph{not} satisfied by the elements of $\text{Hall--Witt}(V_B, 3)$. For instance we have the identity $\phi_{(\{1, 2, 3\},1)}(e_1, e_2, e_3) = 0$, while $\phi_{(\{1, 2, 3\},1)}(e_2, e_1, e_3) = 1$. On the other hand we have that $\phi_{(\{1, 2, 3,4\}, 1)}(e_1, e_2, e_3, e_4)=\phi_{(\{1, 2, 3,4\}, 1)}(e_2, e_1, e_3, e_4) = 0$.
\end{remark}

For $B \subseteq [n]$, we define $\text{Gov}(V_B, i)$ to be the space
$$
\text{Gov}(V_B, i)=\langle \{\phi_{(A, x)}\}_{A \subseteq B, \#A = i, x \in A} \rangle.
$$
The following fact will be crucial for us. We shall give two different proofs for it. 

\begin{proposition} 
\label{Consistent tensors are governing tensors}
Let $i$ be in $\mathbb{Z}_{\geq 1}$ and let $B$ be a subset of $[n]$. Then we have that
$$
\emph{Cons}(V_B, i) = \emph{Gov}(V_B, i).
$$
\end{proposition}

\begin{proof}
For $i = 1$ this is by definition, so from on we assume that $i \geq 2$. We firstly show that $\text{Gov}(V_B, i) \subseteq \text{Cons}(V_B, i)$. For $i = 2$ this is clear, since 
$$
\phi_{(\{j_1, j_2\}, j_1)} = \chi_{j_1} \otimes \chi_{j_2}+\chi_{j_2} \otimes \chi_{j_1},
$$
hence $\phi_{(\{j_1, j_2\}, j_1)} \in \text{Sym}(V_B, 2)$. Next we consider the case $i = 3$. Suppose that $A \subseteq B$ with $\#A=3$ and let $x \in A$. We have to show that $\phi_{(A, x)}$ is in $\text{Hall--Witt}(V_B, 3)$. Thanks to Proposition \ref{P of governing tensors} we see that $P(V_B, 3)(\phi_{(A, x)}) \in \prod_{j \in B}\text{Sym}(V_{B - \{j\}}, 2)$. We next show that $\phi_{(A, x)}$ satisfies the Hall-Witt equation. Write $A := \{y, z, x\}$. Equation (\ref{eGovTen}) shows
$$
\phi_{(A, x)}(e_y, e_z, e_x) + \phi_{(A, x)}(e_z, e_x, e_y) + \phi_{(A, x)}(e_x, e_y, e_z) = 1 + 1 + 0 = 0.
$$
Because $P(V_B, 3)(\phi_{(A, x)}) \in \prod_{j \in B}\text{Sym}(V_{B - \{j\}}, 2)$, this proves that 
$$
\phi_{(A, x)}(e_{\sigma(y)}, e_{\sigma(z)}, e_{\sigma(x)})+\phi_{(A, x)}(e_{\sigma(z)}, e_{\sigma(x)}, e_{\sigma(y)})+\phi_{(A, x)}(e_{\sigma(x)}, e_{\sigma(y)}, e_{\sigma(z)})=0
$$
for any permutation $\sigma$ of $\{x,y,z\}$. Finally observe that if $j \not \in A$, then 
$$
\phi_{(A, x)}(e_j, -, -)=\phi_{(A, x)}(-, e_j, -) = \phi_{(A, x)}(-, -, e_j) = 0. 
$$
Therefore we have verified that $\phi_{(A, x)}$ satisfies all Hall--Witt equations.

We now turn to the case $i \geq 4$. Let $A \subseteq B$ with $\#A = i$ and let $x \in A$. We want to show that $\phi_{(A, x)} \in \text{Cons}(V_B, i)$. Due to Proposition \ref{P of governing tensors} and the inductive step we have that $P(V_B, i)(\phi_{(A, x)}) \in \prod_{j \in B}\text{Cons}(V_{B - \{j\}}, i - 1)$. Next we show that $\phi_{(A, x)}$ satisfies the Commutativity equations. Let $j_1, j_2 \in A$ be distinct. We distinguish two cases. Firstly suppose that $x \in \{j_1, j_2\}$. Then, precisely due to the fact that $i \geq 4$, we see that $\phi_{(A, x)}(e_{j_1}, e_{j_2}, -)=\phi_{(A, x)}(e_{j_2}, e_{j_1}, -)=0$ (see also Remark \ref{asymmetry  3 and 4}). Next suppose that $x \not \in \{j_1, j_2\}$. Then, thanks to Proposition \ref{P of governing tensors} applied twice (which is possible since $i \geq 4$), we have that
\begin{align*}
(\pi_{j_1} \circ &P(V_{B-\{j_1\}}, i - 1) \circ \pi_{j_2} \circ P(V_B, i))(\phi_{(A, x)}) = \phi_{(A - \{j_1, j_2\}, x)} \\
&= (\pi_{j_2} \circ P(V_{B-\{j_2\}}, i - 1) \circ \pi_{j_1} \circ P(V_B, i))(\phi_{(A, x)}).
\end{align*}
Altogether we have established our claim that $\text{Gov}(V_B, i) \subseteq \text{Cons}(V_B, i)$. We next show that $\text{Cons}(V_B, i) \subseteq \text{Gov}(_B, i)$. Thanks to the previous step and Proposition \ref{counting} it suffices to show that 
\begin{align}
\label{eDimIn}
\text{dim}_{\mathbb{F}_2} \text{Cons}(V_B, i) \leq (i - 1) \cdot \binom{\#B}{i}.
\end{align}
We claim that $b \in \text{Cons}(V_B, i)$ is completely determined by its behavior on the set of tuples $(e_{j_1}, \dots, e_{j_i})$ with $j_h \neq j_k$ for distinct $h, k \in [i]$, $j_h \leq j_i$ for each $h \in [i]$, and for every $1 \leq h < k \leq i - 2$ we have $j_h < j_k$. Observe that every subset of $B$ with cardinality $i$ gives precisely $i - 1$ such tuples. We conclude that there are precisely $(i - 1) \cdot \binom{\#B}{i}$ such vectors. Hence the claim implies the desired inequality (\ref{eDimIn}). Therefore our proof is complete once we establish the claim.

To prove the claim, we observe that since $b$ is in $\widetilde{\text{Multi}}(V_B, i)$, $b$ is certainly determined by the tuples $(e_{i_1}, \dots, e_{i_j})$ with $j_h \neq j_k$ for distinct $h, k \in [i]$. Let $(e_{j_1}, \dots, e_{j_i})$ be such a tuple. Thanks to the Commutativity equation applied repeatedly, and the fact that $b$ is in $\text{Cons}(V_B, i)$, we find that we can always assume that for every $1 \leq h < k \leq i - 2$ we have $j_h<j_k$; this does not change the value of $b$. 

Furthermore due to the Symmetry equation we can assume that $j_{i - 1} < j_i$. Therefore we have established the claim in the case $i = 2$, so from now on we assume that $i \geq 3$. Now there are two possibilities. If $j_{i - 2} < j_i$, then our tuple is of the desired form and we are done. Suppose that instead $j_{i - 2} > j_i$. The Hall--Witt and Symmetry equation yield
$$
b(-, e_{j_{i - 2}}, e_{j_{i - 1}}, e_{j_i}) = b(-, e_{j_{i}}, e_{j_{i - 1}}, e_{j_{i - 2}}) + b(-, e_{j_{i - 1}}, e_{j_{i}}, e_{j_{i - 2}}).
$$
Finally we use the commutativity equation once more to rearrange, if needed, the first $i - 2$ entries in $b(-, e_{j_{i - 1}}, e_{j_{i}}, e_{j_{i - 2}})$ to guarantee that the corresponding indices are arranged in monotonic order. Once that is done we have that both summands on the right hand side are evaluations of $b$ in tuples of the desired shape. It follows that $b$ is completely determined by its behavior on such tuples and therefore inequality (\ref{eDimIn}) holds. This concludes the argument. 
\end{proof}

Let $(k_1, \dots, k_n)$ be in $\mathbb{Z}_{\geq 1}^{n}$. We denote by $V_{[k_1, \dots, k_n]}$ the vector space $V_{[k_1]} \times \dots \times V_{[k_n]}$. We have a natural surjective homomorphism $\pi_{(k_1, \dots, k_n)} : V_{[k_1, \dots, k_n]} \twoheadrightarrow V_{[n]}$ obtained by summing each block of $k_i$ coordinates for $i$ in $[n]$. In what follows, wherever the notation suggests so, we are identifying the space $V_{[k_1, \dots, k_n]}$ with the space $V_{k_1+ \dots + k_n}$ by concatenation of coordinates. 

We next define a subspace
$$
\widetilde{\text{Cons}}(V_{[k_1, \dots, k_n]}, i) \subseteq \text{Cons}(V_{k_1+ \dots + k_n}, i),
$$
that will be used in the arithmetical Section \ref{Expansion arithmetic} by means of the intermediate notion of $(k_1, \ldots ,k_n)$-expansion explored in Section \ref{Expansion algebraic}. This will deal with the case of multi-quadratic fields obtained from quadratic fields with composite discriminant. 

\begin{definition} 
\label{furtherly constrained tensors}
Let $i \in \mathbb{Z}_{\geq 2}$ be given. We set $\widetilde{\text{Cons}}(V_{[k_1, \dots, k_n]}, i)$ to be the subgroup of $\text{Cons}(V_{k_1+ \dots + k_n}, i)$ consisting of those $\beta$ in $\text{Cons}(V_{k_1+ \dots + k_n}, i)$ such that
$$
\beta(\sigma_1, \dots, \sigma_{i - 1}, \sigma_{i})=0
$$
whenever we are in one of the following three cases\\
$(1)$ there is $1 \leq h \leq i - 2$ with $\sigma_h \in \text{ker}(\pi_{(k_1, \dots, k_n)})$; \\
$(2)$ both $\sigma_{i - 1}, \sigma_{i}$ are in $\text{ker}(\pi_{(k_1, \dots, k_n)})$; \\
$(3)$ there exist two distinct $h_1, h_2 \in [i]$ and $k \in [n]$ such that 
$$
\pi_{(k_1, \dots, k_n)}(\sigma_{h_1}) = \pi_{(k_1, \dots, k_n)}(\sigma_{h_2}) = e_k.
$$ 
\end{definition}

Recall that Proposition \ref{Consistent tensors are governing tensors} tells us that $\text{Cons}(V_{k_1+ \dots + k_n}, i)$ is generated by governing tensors. Our final step is to show that something analogous is true for the subspace $\widetilde{\text{Cons}}(V_{[k_1, \dots, k_n]}, i)$. To do so we begin by pinpointing the generalized governing tensors. We define a map from $f$ from $[n]$ to the power set of $[k_1 + \dots + k_n]$ by
\begin{align}
\label{ef}
f(i) := [k_1 + \dots + k_i] - [k_1 + \dots + k_{i - 1}].
\end{align}
We say that $j_1, j_2 \in [k_1 + \dots + k_n]$ are in the same block if there is some integer $i$ such that $j_1, j_2 \in f(i)$. If $T \subseteq [k_1 + \dots + k_n]$, we extend $f$ by defining $f(T) = T$. We warn the reader that with this notation we have that $f(i)$ need not equal $f(\{i\})$. For $A \subseteq [n], \#A = i \geq 2, x \in A$ and $T \subseteq f(x)$ we put
$$
\widetilde{\phi}_{(A, T)} := \sum_{\substack{\tau \in \text{Isom}_{\text{Set}}([i], (A - \{x\}) \cup \{T\}) \\ \tau(i - 1) = T \ \text{or} \ \tau(i) = T}} \Bigg (\sum_{x_1 \in f(\tau(1))}\chi_{x_1} \Bigg) \otimes \dots \otimes \Bigg(\sum_{x_i \in f(\tau(i))}\chi_{x_i} \Bigg),
$$
while for $i = 1$ we put $\widetilde{\phi}_{(A, T)} := \sum_{x_i \in T} \chi_{x_i}$. We denote for $i \in \mathbb{Z}_{\geq 1}$
$$
\widetilde{\text{Gov}}(V_{[k_1, \dots, k_n]}, i) := \langle \{\widetilde{\phi}_{(A, T)} \}_{A \subseteq [n] \ \text{with} \ \#A=i, \ x \in A, \ T \subseteq [f(x)]} \rangle
$$
and $\widetilde{\text{Cons}}(V_{[k_1, \dots, k_n]}, 1) := \widetilde{\text{Gov}}(V_{[k_1, \dots, k_n]}, 1)$. We have the following elementary fact which generalizes Proposition \ref{counting}.

\begin{proposition} 
\label{counting more}
Let $i$ be in $\mathbb{Z}_{\geq 2}$. Then we have that
$$
\emph{dim}_{\mathbb{F}_2} \widetilde{\emph{Gov}}(V_{[k_1, \dots, k_n]}, i) = (k_1+ \dots + k_n) \cdot \binom{n - 1}{i - 1} - \binom{n}{i}.
$$
\end{proposition} 

\begin{proof}
It is clear that the space is spanned already by $\widetilde{\phi}_{(A, T)}$ with $\#T=1$. Hence for each choice of $A$ we have $\sum_{s \in A} k_s$ choices of $T$. Among these tensors, obtained from a choice of $A$, there is precisely one relation for the same argument as outlined in the proof of Proposition \ref{counting}.
Furthermore the various spaces obtained as $A$ varies are jointly in direct sum. Therefore we get that 
$$
\text{dim}_{\mathbb{F}_2} \widetilde{\text{Gov}}(V_{[k_1, \dots, k_n]}, i) = \sum_{A \subseteq [n], \#A=i} \left(\left(\sum_{s \in A} k_s\right) - 1\right).
$$
In the total sum each $k_s$ appears $\binom{n - 1}{i - 1}$ many times and the $-1$ appears $\binom{n}{i}$ times. Therefore we get the desired equality. 
\end{proof}

The following fact provides a generalization of Proposition \ref{Consistent tensors are governing tensors}. At this point, it can be proved in three different manners. A first way is to use Proposition \ref{Consistent tensors are governing tensors} in an essential manner; one can show explicitly that the subspace of $\text{Gov}(V_{k_1+ \dots + k_n}, i)$ cut out by the additional requirements of $\widetilde{\text{Cons}}(V_{[k_1, \dots, k_n]}, i)$ is precisely $\widetilde{\text{Gov}}(V_{[k_1, \dots, k_n]}, i)$. The second and third way are to directly transfer the steps of respectively the first and the second proof of Proposition \ref{Consistent tensors are governing tensors} to the present context. We shall only give the third proof.
 
\begin{proposition} 
\label{Consistent tensors are governing tensors more general}
Let $i$ be in $\mathbb{Z}_{\geq 1}$. Then we have that
$$
\widetilde{\emph{Cons}}(V_{[k_1, \dots, k_n]}, i) = \widetilde{\emph{Gov}}(V_{[k_1, \dots, k_n]}, i).
$$
\end{proposition}

\begin{proof}
For $i = 1$ this is by definition. Henceforth we shall assume $i \geq 2$. We start off by proving that
\begin{align}
\label{eIncMult}
\widetilde{\text{Gov}}(V_{[k_1, \dots, k_n]}, i) \subseteq \widetilde{\text{Cons}}(V_{[k_1, \dots, k_n]}, i).
\end{align}
Thanks to Proposition \ref{Consistent tensors are governing tensors} we already know that $\widetilde{\text{Gov}}(V_{[k_1, \dots, k_n]}, i) \subseteq \text{Cons}(V_{k_1+ \dots + k_n}, i)$. We still need to verify that for each set $A$ with $\#A = i$, $\widetilde{\phi}_{(A, T)}$ vanishes on all tuples $(\sigma_1, \dots, \sigma_i)$ satisfying at least one of the three conditions listed in Definition \ref{furtherly constrained tensors}. In all three cases one finds as a result of a direct inspection that the desired vanishing takes place even termwise for the right hand side of the defining equation of $\widetilde{\phi}_{(A, T)}$. This establishes equation (\ref{eIncMult}).

Due to Proposition \ref{counting more} it remains to prove that
\begin{align}
\label{eDimIn2}
\text{dim}_{\mathbb{F}_2} \widetilde{\text{Cons}}(V_{[k_1, \dots, k_n]}, i) \leq (k_1+ \dots + k_n) \cdot \binom{n - 1}{i - 1} - \binom{n}{i}
\end{align}
To this end, let $g: [n] \rightarrow [k_1 + \dots k_n]$ be the function $g(s) := k_1 + \dots + k_{s - 1} + 1$. Let $b \in \widetilde{\text{Cons}}(V_{[k_1, \dots, k_n]}, i)$. We claim that $b$ is determined by its value on the set of tuples $(e_{r_1}, \dots, e_{r_i})$ with the following properties
\begin{enumerate}
\item[(a)] $r_1, \dots, r_i$ are pairwise in different blocks;
\item[(b)] every $r_h$ with $h$ not in $\{i - 1, i\}$ is in the image of $g$;
\item[(c)] at least one among $r_{i - 1}$ and $r_i$ is in the image of $g$;
\item[(d)] the first $i-2$ values of the indices are ordered in a strictly increasing manner and furthermore $r_j < r_i$ for every $j$ with $1 \leq j \leq i-1$;
\item[(e)] in case $r_i$ is not in the image of $g$ then $r_j<r_{i-1}$ for every $j$ with $1 \leq j \leq i-2$.
\end{enumerate}

Observe that each subset $A \subseteq [n]$ with $\#A=i$ gives a collection of $(\sum_{s \in A} k_s) - 1$ such tuples, by demanding that for each $j \in [i]$, the value $r_j$ belongs to a block whose index is in $A$. Indeed, one has $\sum_{s < \text{max}(A)}k_s$ choices when the $i$-th entry is in the image of $g$ and $k_{\text{max}(A)}-1$ when the last entry is not in the image of $g$, thanks to the last requirement. Therefore the claim implies inequality (\ref{eDimIn2}) by the same argument as in Proposition \ref{counting more}, and hence the proposition. 

We now prove the claim. Thanks to rule $(3)$ of Definition \ref{furtherly constrained tensors} we can assume that each of the indices $r_j$ belongs to a different block. Next, since $\widetilde{\text{Cons}}(V_{[k_1, \dots, k_n]}, i)$ is in particular a subspace of $\text{Cons}(V_{k_1+ \dots+ k_n}, i)$, we can always assume that the first $i-2$ values of $r_j$ are given in a strictly increasing fashion. Furthermore, thanks to rule $(1)$ of Definition \ref{furtherly constrained tensors}, we can assume that for each of the first $i-2$ values of $j$ the element $r_j$ is in the image of $g$. 

We claim that at least one of $r_{i-1}$ and $r_i$ can be assumed to be in the image of $g$. Indeed suppose that is not the case. Let $v_{i-1}$ and $v_i$ be respectively the unique basis vectors in the block of $r_{i - 1}$ and $r_i$ that is in the image of $g$. Now thanks to rule $(2)$ of Definition \ref{furtherly constrained tensors} we see that $b( \ldots, e_{r_{i-1}}+v_{i-1},e_{r_i}+v_i)=0$. After expanding this, we get that the desired value of $b$ can be expressed as the sum of $3$ values with the property that at least one between the $i-1$-th and the $i$-th entry is in the image of $g$. This shows the claim.

Summarizing, we have shown that the vector can be taken with indices belonging to all different blocks, the first $i-2$ entries with indices ordered in a strictly increasing fashion, all in the image of $g$, and the last two entries having at least one of the two indices in the image of $g$. Now applying the Symmetry equation we can also assume that $r_{i-1}<r_i$. We next reduce to the case where $r_i > r_j$ for each $j \in [i-1]$. For this we can assume assume that $i \geq 3$, otherwise we are already done. Since the first $i-2$ are ordered, we only need to ensure that $r_i>r_{i-2}$. Suppose not, then applying the Hall--Witt equation yields
$$
b(\ldots, e_{r_{i-2}},e_{r_{i-1}},e_{r_i})=b(\ldots, e_{r_i},e_{r_{i-2}},e_{r_{i-1}})+b(\ldots, e_{r_{i-1}},e_{r_i},e_{r_{i-2}}).
$$
Now observe that in both summands the largest entry, $r_{i-2}$, is among the last two coordinates, and hence as an application of the Symmetry equation can be assumed to be the last entry. Also one of the two terms $b(\ldots, e_{r_i},e_{r_{i-2}},e_{r_{i-1}})$ or $b(\ldots, e_{r_{i-1}},e_{r_i},e_{r_{i-2}})$ might have the $i-2$-th entry not in the image of $g$. However we can apply rule $(1)$ of Definition \ref{furtherly constrained tensors} to fix this. Finally we can reapply the Commutativity equation, to get $r_i$ and $r_{i-1}$ in their proper placement among the first $i-2$ indices in the respective terms $b(\ldots, e_{r_i},e_{r_{i-2}},e_{r_{i-1}})$ and $b(\ldots, e_{r_{i-1}},e_{r_i},e_{r_{i-2}})$. 

Summarizing once more, we have shown that the vector can be taken with indices belonging to all different blocks, the first $i-2$ entries with indices ordered in a strictly increasing fashion, all in the image of $g$, the last two entries having one of the two indices in the image of $g$ and the index $r_i$ is larger than $r_j$ for all $j$ in $[i-1]$. Now in case $r_i$ is in the image of $g$ this becomes one of the listed vectors and hence we are done. 

So assume that $r_i$ is not in the image of $g$. Then $r_{i-1}$ must be in the image of $g$, since we know that at least one of the two is in the image of $g$. Then in case $r_{i-1}$ is larger than $r_j$ for each $j$ in $[i-2]$ we are also dealing with one of the listed vectors and hence we are done. Therefore we can assume that $r_{i-1}$ is not maximal among the first $i-1$ indices. Since the first $i-2$ indices are ordered, this amounts to $r_{i-1} < r_{i-2}$. We apply again the Hall--Witt equation to obtain that 
$$
b(\ldots, e_{r_{i-2}},e_{r_{i-1}},e_{r_i})=b(\ldots, e_{r_i},e_{r_{i-2}},e_{r_{i-1}})+b(\ldots, e_{r_{i-1}},e_{r_i},e_{r_{i-2}}).
$$
We see that $b(\ldots, e_{r_{i-1}},e_{r_i},e_{r_{i-2}})$ is an evaluation at a vector with the desired shape after we apply the Symmetry equation to fix the order of the last two indices and the Commutativity equation to get $r_{i-1}$ in the right placement among the first $i-2$ indices. We now focus on the first term. We can apply rule $(1)$ from Definition \ref{furtherly constrained tensors} to rewrite the evaluation  $b(\ldots, e_{r_i},e_{r_{i-2}},e_{r_{i-1}})=b(\ldots, e_{r_i'},e_{r_{i-2}},e_{r_{i-1}})$, where $r_i'$ is the unique index in the same block of $r_i$ and in the image of $g$. We now only need to handle the fact that the largest term is in position $i-2$. To do so we apply one more time on this last term the Hall--Witt equation to obtain 
$$
b(\ldots, e_{r_i^{'}},e_{r_{i-2}},e_{r_{i-1}})=b(\ldots, e_{r_{i-2}},e_{r_{i-1}},e_{r_{i}^{'}})+b(\ldots, e_{r_{i-1}},e_{r_{i}^{'}},e_{r_{i-2}}).
$$
Now the first term is one of the listed evaluations, so we need to only focus on the last term. We can apply the Symmetry equation to swap the last two entries and the Commutativity equation to put $r_{i-1}$ in its proper placement among the first $i-2$ indices. The resulting evaluation is among the listed ones. This completes the proof.
\end{proof}
 
\section{Expansion groups and expansion Lie algebras} \label{Expansion algebraic}
For the remainder of this section let $n$ be a positive integer and let $(k_1, \dots, k_n)$ be a vector in $\mathbb{Z}_{\geq 1}^{n}$. The goal of this section is to analyze the category of $n$-\emph{expansion groups} and more generally $[(k_1, \dots, k_n)]$-\emph{expansion groups}. To do so, we will introduce and analyze the category of $n$-\emph{expansion Lie algebras} and more generally $[(k_1, \dots, k_n)]$-\emph{expansion Lie algebras}. As we shall see in Section \ref{Expansion arithmetic}, these are algebraic structures that capture the essential properties satisfied by the Galois groups that occur in higher genus theory. 

The central results of this section are classification theorems for these categories in terms of universal objects, which are naturally built using the tensors spaces of Section \ref{Auxiliary tensor spaces}: the main step of the classification theorems will be an application of Proposition \ref{Consistent tensors are governing tensors} for $n$-expansions and Proposition \ref{Consistent tensors are governing tensors more general} for $[(k_1, \dots, k_n)]$-expansions, which, therefore, form the technical heart behind the proofs of the classification theorems. 

The spaces $\text{Cons}(V_{[n]}, j)$ will be crucial to handle $n$-expansions while the more general spaces $\widetilde{\text{Cons}}(V_{[k_1, \dots, k_n]}, j)$ will be crucial to handle $[(k_1, \dots, k_n)]$-expansions. Recall that $\widetilde{\text{Cons}}(V_{[k_1, \dots, k_n]}, j)$ is a special subspace of $\text{Cons}(V_{[k_1, \dots, k_n]}, j)$. Similarly, a $[(k_1, \dots, k_n)]$-expansion is by definition a $(k_1+ \dots + k_n)$-expansion with extra conditions. For this reason we firstly deal with the more fundamental notion of $n$-expansion and we then explain how to handle the slightly refined notion of $[(k_1, \dots, k_n)]$-expansion. 

When we apply the material of this section to the arithmetical Section \ref{Expansion arithmetic}, this way of proceeding, once unwrapped, boils down to first establishing higher genus theory for multi-quadratic number fields obtained as the compositum of quadratic number fields of prime discriminant and then proceed with the general case by means of a reduction to the prime case.

\subsection{\texorpdfstring{$n$-expansion groups and $n$-expansion Lie algebras}{n-expansion groups and n-expansion Lie algebras}}
\label{n-expansion and $n$-expansion Lie algebras}
We begin with the definition of an $n$-expansion group. Recall from the beginning of Section \ref{Auxiliary tensor spaces} the definition of $V_{[n]}$ and of $\{e_i\}_{i \in [n]}$. 

\begin{definition} 
\label{Expansion of groups}
We call a triple $(G, \phi, (g_1, \dots, g_n))$ an $n$-\emph{expansion group} if the following properties hold 
\begin{enumerate}
\item $G$ is a group, $\phi$ is a homomorphism from $G$ to $V_{[n]}$, $g_i \in G$ and $\phi(g_i)=e_i$; 
\item $\text{ker}(\phi)$ is a vector space over $\mathbb{F}_2$;
\item we have that $[G, G] = \text{ker}(\phi)$;
\item we have that $g_i^{2} = \text{id}$ for each $i \in [n]$. 
\end{enumerate}
\end{definition}

To examine $n$-expansion groups it is convenient to first introduce and examine $n$-expansion Lie algebras. Recall the following well-known procedure to attach a graded Lie algebra to a group $G$. For each $m \in \mathbb{Z}_{\geq 1}$ we set
$$
L_m(G) := \frac{G^{(m)}}{G^{(m+1)}},
$$
where for every positive integer $i$ the group $G^{(i)}$ denotes the $i$-th term of the descending central series of $G$ defined recursively as $G^{(1)}= G$ and $G^{(i + 1)}= [G, G^{(i)}]$ for every positive integer $i$. Taking commutators in $G$ induces a bi-linear operator $[-, -]_{G}$ on
$$
L_{\bullet}(G) := \bigoplus_{m \in \mathbb{Z}_{\geq 1}} L_m(G).
$$ 
It is easy to verify that the operator is alternating and the content of the Hall-Witt identity is translated in the Jacobi identity, i.e. $(L_{\bullet}(G), [,]_{G})$ is a Lie algebra. Furthermore the algebra $L_{\bullet}(G)$ is \emph{graded}, meaning that $[L_h(G), L_k(G)]_{G} \subseteq L_{h+ k}(G)$ for each $h,k \in \mathbb{Z}_{\geq 1}$. For each group homomorphism $f:G_1 \to G_2$ we denote the naturally corresponding graded Lie algebra homomorphism with $L(f):L_{\bullet}(G_1) \to L_{\bullet}(G_2)$. We conclude that the assignment $G \mapsto L_{\bullet}(G)$ is a functor.

We call a graded Lie algebra\footnote{The notation $L_{\bullet}$ will always stand for a graded Lie algebra and will implicitly give the notation $L_m$, where for a positive integer $m$ the group $L_m$ is the $m$-th piece of the grading and $L_{\bullet} = \bigoplus_{m \in \mathbb{Z}_{\geq 1}}L_m$.} $(L_{\bullet}, [,])$ a graded Lie algebra over $\mathbb{F}_2$ in case $L_{\bullet}$ is a vector space over $\mathbb{F}_2$. We view the vector space $V_{[n]}$ as a graded abelian Lie algebra completely concentrated in degree $1$ by equipping it with the zero map as Lie bracket.
 
If $L$ is a Lie algebra, then we can define the descending central series for $L$ to be the sequence of $L^{(i)}$, defined for each $i \in \mathbb{Z}_{\geq 1}$, in the following recursive manner. We put $L^{(1)} = L$ and $[L, L^{(i)}] = L^{(i + 1)}$ for every positive integer $i$. For a positive integer $m$ we say that $L$ is $m$-nilpotent if $L^{(m)} = 0$. Furthermore we say that $L$ is $\omega$-nilpotent in case $\bigcap_{m \in \mathbb{Z}_{\geq 1}}L^{(m)} = \{0\}$. Observe that, since our gradings always start with $1$, any graded Lie algebra is automatically $\omega$-nilpotent. As we shall see in Proposition \ref{Lie algebra of an expansion is an expansion of a Lie algebra}, Definition \ref{Expansion of groups} has the following Lie algebra analogue.
 
\begin{definition} 
\label{Expansion of a Lie algebra}
We call a pair $((L_{\bullet}, [,]), \psi)$ an $n$-\emph{expansion Lie algebra} if the following properties hold 
\begin{enumerate}
\item $(L_{\bullet}, [,])$ is a graded Lie algebra over $\mathbb{F}_2$ and $\psi$ is a surjective homomorphism of graded Lie algebras from $L_{\bullet}$ to $V_{[n]}$;
\item $(\text{ker}(\psi), [,])$ is an abelian sub-algebra;
\item $[L_{\bullet}, L_{\bullet}] =\text{ker}(\psi)$;
\item for each $i \in \mathbb{Z}_{\geq 4}$ and for each $\sigma_1, \dots, \sigma_i$ in $L_1$ the commutator 
$$
[\sigma_{1}, [\sigma_{2}, [\dots, [\sigma_{i - 1}, \sigma_{i}] \dots ]]]
$$ 
does not depend on the order of the first $i - 2$ entries and vanishes as soon as there are distinct $s,t \in [i - 2]$ with $\sigma_s=\sigma_t$. Furthermore if $\tau_1, \tau_2,$ are in $L_1$ and $\psi(\tau_1)$ is in $\{e_1, \dots, e_n\}$, then $[\tau_1, [\tau_1, \tau_2]] =0$. 
\end{enumerate}
\end{definition}

The following fact explains the connection between Definition \ref{Expansion of groups} and Definition \ref{Expansion of a Lie algebra}. 

\begin{proposition} 
\label{Lie algebra of an expansion is an expansion of a Lie algebra}
Suppose that an $n$-expansion group $(G, \phi,(g_1, \dots, g_n))$ is given. Then $((L_{\bullet}(G), [,]_{G}), L_{\bullet}(\phi))$ is an $n$-expansion Lie algebra. 
\end{proposition}

\begin{proof}
We claim that $(L_{\bullet}(G), [,]_{G})$ is a graded Lie algebra over $\mathbb{F}_2$ and $L_{\bullet}(\phi)$ an epimorphism of graded Lie algebras. Indeed, $[G, G] = G^{(2)}$ is a vector space over $\mathbb{F}_2$ thanks to axiom $(2)$ and $(3)$ of Definition \ref{Expansion of groups} and $L_1(G)=\frac{G}{G^{(2)}}$ is also a vector space over $\mathbb{F}_2$ due to axiom $(1)$ of Definition \ref{Expansion of groups}. Finally observe that the map $L_{\bullet}(\phi)$ is an isomorphism between $L_1(G)$ and $V_{[n]}$, which sends every other graded piece to $0$. Hence $L_{\bullet}(\phi)$ is an epimorphism of graded Lie algebras, thus establishing our claim.

Next, thanks to axiom $(2)$ of Definition \ref{Expansion of groups}, we have that every commutator in $G^{(2)}$ vanishes. Therefore $\oplus_{i \in \mathbb{Z}_{\geq 2}}L_i$ is an abelian subalgebra, which is also the kernel of $L_{\bullet}(\phi)$. Therefore axiom $(2)$ is verified. Similarly, axiom $(3)$ is an immediate reformulation of axiom $(3)$ of Definition \ref{Expansion of groups}. 

We finally show that axiom $(4)$ is satisfied for $((L_{\bullet}(G), [,]_{G}), L_{\bullet}(\phi))$. To do so, let $ \widetilde{\tau_1}, \widetilde{\tau_2} \in \{g_1, \dots, g_n\}$ be the involutions given by axiom $(4)$ of Definition \ref{Expansion of groups}. Observe that $[\widetilde{\tau_1}, \widetilde{\tau_2}]^{2} = \text{id}$ thanks to axiom $(2)$ and $(3)$ of Definition \ref{Expansion of groups}. Furthermore we know that $\widetilde{\tau_1}^2 = \widetilde{\tau_2}^2 = \text{id}$. It follows that
\begin{align*}
[\widetilde{\tau_1}, [\widetilde{\tau_1}, \widetilde{\tau_2}]] &= \widetilde{\tau_1}[\widetilde{\tau_1}, \widetilde{\tau_2}]\widetilde{\tau_1}[\widetilde{\tau_1}, \widetilde{\tau_2}]
= \widetilde{\tau_1}\widetilde{\tau_1}\widetilde{\tau_2} \widetilde{\tau_1} \widetilde{\tau_2} \widetilde{\tau_1} \widetilde{\tau_1}\widetilde{\tau_2} \widetilde{\tau_1} \widetilde{\tau_2} \\
&= \widetilde{\tau_2} \widetilde{\tau_1} \widetilde{\tau_2} \widetilde{\tau_2} \widetilde{\tau_1} \widetilde{\tau_2}
= \widetilde{\tau_2} \widetilde{\tau_1} \widetilde{\tau_1} \widetilde{\tau_2}
= \widetilde{\tau_2}  \widetilde{\tau_2}=\text{id}.
\end{align*}
Now suppose firstly that $\tau_1, \tau_2 \in L_1(G)$ with $\psi(\tau_1), \psi(\tau_2) \in \{e_1, \dots, e_n\}$. We can find in the set $\{g_1, \dots, g_n\}$ lifts $\widetilde{\tau_1}, \widetilde{\tau_2}$ of $\tau_1$ and $\tau_2$, so that $\widetilde{\tau_1}, \widetilde{\tau_2}$ are involutions. Hence the calculation above shows 
\begin{align}
\label{etau}
[\tau_1, [\tau_1, \tau_2]] = 0. 
\end{align}
We claim that this implies the last part of axiom $(4)$ from Definition \ref{Expansion of a Lie algebra}. Indeed, on the one hand the assignment $\sigma \mapsto [\tau_1,[\tau_1,\sigma]]$ is a linear functional from $L_1(G)$ to $L_3(G)$. On the other hand we have shown that it vanishes for any choice of $\sigma:=\tau_2$ in the classes of $\{g_1, \dots, g_n\}$. However such classes form a basis of $L_1(G)$, as we can see combining axioms $(1),(2)$ and $(3)$ from Definition \ref{Expansion of groups}. Therefore the functional $[\tau_1,[\tau_1,-]]$ is the trivial one and the validity of the last part of axiom $(4)$ for $(L_{\bullet}(G), L_{\bullet}(\phi))$ is established.

Next observe that since $\text{ker}(\phi) = [G, G]$ is abelian, the action of $G$ on $[G, G]$ factors completely modulo $[G, G]$. Since $[G, G]$ is a $\mathbb{F}_2$-vector space, it naturally becomes a module over the group ring $\mathbb{F}_2[\frac{G}{[G, G]}]$. Therefore we have for all $\widetilde{\sigma_1}, \dots, \widetilde{\sigma_i} \in G$
$$
[\widetilde{\sigma_{1}}, [\widetilde{\sigma_{2}}, [\dots, [\widetilde{\sigma_{i - 1}}, \widetilde{\sigma_{i}}] \dots ]]] = (1 + \widetilde{\sigma_1}) \cdot \ldots \cdot (1+\widetilde{\sigma_{i - 2}})([\widetilde{\sigma_{i - 1}}, \widetilde{\sigma_i}]).
$$
Here the elements $1+\widetilde{\sigma_1}, \dots,1+\widetilde{\sigma_i}$ are taken in the group ring $\mathbb{F}_2[\frac{G}{[G, G]}]$. This group ring is commutative and isomorphic to $\mathbb{F}_2[\mathbb{F}_2^{n}]$. In this ring every element that is not a unit has square equal to $0$. Now let $\sigma_1, \dots, \sigma_i$ be in $L_1(G)$. We can lift $\sigma_1, \dots, \sigma_i$ to elements $\widetilde{\sigma_1}, \dots, \widetilde{\sigma_i}$ in $G$. Then, by the calculation we have just done, we see that the value of the bracket is independent of the order of the first $i - 2$ elements thanks to commutativity of the group ring. 

Furthermore, we see that if two entries are equal, we can rearrange the order in such a way that $(1+\widetilde{\sigma_1}) \cdot \ldots \cdot (1+\widetilde{\sigma_{i - 2}})$ contains a square of an element in the augmentation ideal (hence not a unit), so the product is $0$. This concludes the proof. 
\end{proof}

We shall need the following fact, which will guarantee that an $n$-expansion can be completely recovered from a naturally attached multi-linear structure on $V_{[n]}$. From now on we will frequently abbreviate $(L_{\bullet}, [,])$ simply as $L_{\bullet}$.

\begin{proposition} 
\label{All comes from L_1}
Let $(L_{\bullet}, \psi)$ be an $n$-expansion Lie algebra. Then for each $i \in \mathbb{Z}_{\geq 1}$ we have that $L_i$ is the span of $[\sigma_{1}, [\sigma_{2}, [\dots, [\sigma_{i - 1}, \sigma_{i}] \dots ]]]$ with $\sigma_1, \dots, \sigma_i$ varying in $L_1$. Furthermore, the map $\psi$ induces an isomorphism between $L_1$ and $V_{[n]}$. 
\end{proposition}

\begin{proof}
We show this by induction on $i$. For $i = 1$ the statement is vacuously true. Assuming that the statement holds for $i$, we shall prove it for $i + 1$. Firstly, we claim that $L_{i + 1}$ is spanned by commutators of the form $[\sigma, \tau]$ with both $\sigma$ and $\tau$ homogeneous elements. 

Indeed, since $L_{\bullet}$ is an $n$-expansion it follows that $L_{i + 1}$ is in $[L_{\bullet}, L_{\bullet}]$. Hence every element of $L_{i + 1}$ will be in the span of commutators. We now expand every entry of every commutator in its homogeneous components. In this way we obtain, after using the bi-linearity of $[-, -]$, a linear combination of commutators with homogeneous entries. By definition of a grading, the terms with total degree different from $i + 1$ must sum to $0$ in order to land in $L_{i + 1}$. 

Next, since $L_{\bullet}$, is an $n$-expansion, it follows from axiom $(2)$ and axiom $(3)$ that at least one of the elements $\sigma,\tau$ above can be assumed to be in $L_1$; otherwise $[\sigma,\tau]$ is $0$. Since $L_{\bullet}$ is a vector space over $\mathbb{F}_2$, we have
\[
[\sigma, \tau] = [\tau, \sigma],
\]
so we can assume that $\sigma$ is of degree $1$. Therefore $\tau$ has degree precisely equal to $i$, since the sum of the degrees is $i + 1$. Hence by the inductive assumption we can rewrite $\tau$ as linear combination of nested commutators with entries purely of degree $1$. After expanding with multi-linearity, we obtain that $[\sigma, \tau]$ can be rewritten as linear combination of nested commutators with all entries purely of degree $1$. We conclude that $L_{i + 1}$ is spanned by such elements, which gives precisely the first part of the proposition.

We next observe that the homomorphism $\psi$ preserves the grading. Therefore it must send, by definition, all the $L_i$ with $i \geq 2$ to $0$. Since $\psi$ is surjective, $\psi|_{L_1}$ is also surjective. We claim that it is also injective. Indeed thanks to axiom $(3)$ we know in particular that $\text{ker}(\psi) \subseteq [L_{\bullet},L_{\bullet}]$. But, since our gradings start from the index $1$, we must have $[L_{\bullet}, L_{\bullet}] \subseteq \oplus_{i \in \mathbb{Z}_{\geq 2}} L_i$. This tells us that $\psi$ restricted to $L_1$ is an injective map. In total we have shown that $\psi|_{L_1}$ is an isomorphism. 
\end{proof}

Due to Proposition \ref{All comes from L_1} we see that an $n$-expansion Lie algebra $(L_{\bullet}, \psi)$ naturally comes with a \emph{surjective} homomorphism
$$
\psi_i: (V_{[n]})^{\otimes i} \to L_i
$$
for every $i \in \mathbb{Z}_{\geq 1}$ defined by
$$
v_1 \otimes \dots .\otimes v_i \mapsto [\psi|_{L_1}^{-1}(v_1), [\psi|_{L_1}^{-1}(v_2), [\dots, [\psi|_{L_1}^{-1}(v_{i - 1}), \psi|_{L_1}^{-1}(v_i)] \dots ]]].
$$ 
For $i=1$ this is the identification $\psi|_{L_1}^{-1}$ between $V_{[n]}$ and $L_1$. Therefore we have a natural injective map for each $i$ in $\mathbb{Z}_{\geq 1}$
$$
\psi_i^{\vee} : L_i^{\vee} \to \text{Multi}(V_{[n]}, i).
$$

\begin{proposition}
\label{expansion bracket give consistent tensors}
We have the inclusion
$$
\emph{Im}(\psi_i^{\vee}) \subseteq \emph{Cons}(V_{[n]}, i).
$$
\end{proposition}

\begin{proof}
Axiom $(4)$ shows that the nested commutator will vanish if there is some $i \in [n]$ such that two entries are equal to $e_i$. We deduce that 
$$
\text{Im}(\psi_i^{\vee}) \subseteq \widetilde{\text{Multi}}(V_{[n]}, i). 
$$
Next observe that the Symmetry equation and the Hall--Witt equation are automatically guaranteed by the fact that the bracket is symmetric (we have algebras over $\mathbb{F}_2$) and that the bracket satisfies the Jacobi identity. Finally, the Commutativity equation is guaranteed precisely by the first part of axiom $(4)$. 
\end{proof}

For every $i \in \mathbb{Z}_{\geq 1}$ we put 
$$
C([n], i) := \text{Cons}(V_{[n]}, i)^{\vee}.
$$
We put
$$
C([n], \bullet) := \bigoplus_{i \in \mathbb{Z}_{\geq 1}} C([n], i),
$$  
and we next turn $C([n], \bullet)$ into a graded Lie algebra. Firstly, we declare the bracket of any two graded pieces of degree at least $2$ to be $0$. So we need to define only the bracket between $C([n],1) = V_{[n]}$ and $C([n], i)$ for each $i \in \mathbb{Z}_{\geq 1}$. Let $\rho \in C([n], i)$, $j \in [n]$ and $\phi \in \text{Cons}(V_{[n]}, i + 1)$. Recall that $\phi(e_j, -)$ is in $\text{Cons}(V_{[n]-\{j\}}, i)$ and hence we can view it as an element of $\text{Cons}(V_{[n]}, i)$ by composing with the natural projection map. We then define 
$$
[e_j, \rho]_{[n]}(\phi):=\rho(\phi(e_j, -)).
$$
This defines $[,]_{[n]}$ by linearity. Altogether we have a natural map of graded vector spaces
$$
\psi_{\bullet} : C([n], \bullet) \to L_{\bullet}.
$$
Also we have a natural projection map $c_{[n]}:C([n], \bullet) \twoheadrightarrow V_{[n]}$. Before stating the next proposition we need a definition. Let $((L_{\bullet}, [,]), \psi)$ and $((L'_{\bullet}, [,]), \psi')$ be two $n$-expansion Lie algebras. A homomorphism of $n$-expansion Lie algebras
$$
f: L_{\bullet} \to L'_{\bullet},
$$
is a homomorphism of graded Lie algebras such that $\psi' \circ f=\psi$. Observe that Proposition \ref{All comes from L_1} easily implies that there is at most one homomorphism between any two $n$-expansion Lie algebras. Furthermore, if there is a homomorphism, then it must be an epimorphism. However, we shall add the adjective surjective precisely to stress this information. The following proposition shows that all $n$-expansion Lie algebras are canonically a quotient of $((C([n], \bullet), [,]_{[n]}), c_{[n]})$.
 
\begin{proposition} 
\label{Construction of universal algebra}
$\emph{(a)}$ The graded vector space $C([n], \bullet)$ equipped with $[,]_{[n]}$ is an $n+1$-nilpotent graded Lie algebra. The pair $((C([n], \bullet), [,]_{[n]}), c_{[n]})$ is an $n$-expansion Lie algebra. \\ $\emph{(b)}$ For every $n$-expansion Lie algebra $((L_{\bullet}, []), \psi)$ we have that 
$$
\psi_{\bullet}:C([n], \bullet) \twoheadrightarrow L_{\bullet}
$$ is a surjective homomorphism of $n$-expansion Lie algebras.
\end{proposition}

\begin{proof}
Part $(a)$ is straightforward. Part $(b)$ follows from Proposition \ref{expansion bracket give consistent tensors}.
\end{proof}

Our next goal is to turn Proposition \ref{Consistent tensors are governing tensors} into a more explicit presentation of the algebra $((C([n], \bullet), [,]_{[n]}), c_{[n]})$ and hence of every other $n$-expansion Lie algebra, by applying Proposition \ref{Construction of universal algebra}. Indeed observe that Proposition \ref{Consistent tensors are governing tensors} already provides us with the non-trivial result that 
$$
\text{dim}_{\mathbb{F}_2}C([n], \bullet)=n2^{n - 1}-2^n+n+1.
$$
In particular it gives us an upper bound for the size of any $n$-expansion Lie algebra and furthermore it shows that its nilpotency class is always at most $n+1$, while a priori the definition trivially guarantees only nilpotency class $\omega$. Such a bound follows from the fact that it is very easy to count the dimension of the spaces $\text{Gov}(\mathbb{F}_2^{[n]}, i)$: the work, done in the proof of Proposition \ref{Consistent tensors are governing tensors}, is in proving that this subspace of $\text{Cons}(\mathbb{F}_2^{[n]}, i)$ is the full space. Hence, our next step will be to put a natural structure of $n$-expansion Lie algebra on the graded vector space over $\mathbb{F}_2$ given by
$$
\text{G}([n], \bullet) := \bigoplus_{i \in \mathbb{Z}_{\geq 1}} \text{G}([n], i),
$$
where $\text{G}([n], i):=\text{Gov}(\mathbb{F}_2^{[n]}, i)^{\vee}$. We see that $\text{G}([n],1)$ can naturally be identified with $V_{[n]}$ and for $i \geq 2$ we have that $\text{G}([n], i)$ can naturally be identified as a subspace 
$$
\text{G}([n], i) \subseteq \mathbb{F}_2^{{(A, x)}_{A \subseteq [n], \#A = i, x \in A}}
$$
of the space of formal $\mathbb{F}_2$-linear combinations of pointed subsets of $[n]$ with cardinality $i$, defined as follows
$$
\text{G}([n], i) := \left\{\sum_{\substack{(A, x) \\ A \subseteq [n], \#A = i, x \in A}} \hspace{-0.5cm} \lambda_{(A, x)}e_{(A, x)}: \sum_{x \in A} \lambda_{(A, x)}=0 \text{ for all } A \subseteq [n] \text{ with } \#A = i \right\}.
$$
For each non-empty $A \subseteq [n]$ and for each $x, y$ in $A$, we have that $e_{(A, x)} + e_{(A, y)}$ is in $\text{G}([n], \#A)$. Furthermore, such elements together with the $e_j$, with $j \in [n]$ form a generating set for $\text{G}([n], \bullet)$. In other words we see that we can rewrite
$$
\text{G}([n], \bullet) = \bigoplus_{\emptyset \neq A \subseteq [n]} W([n], A),
$$
where $W([n], A)$ is the span of the elements $e_{(A,x)}+e_{(A,y)}$ with $x,y$ in $A$ in case $\#A \geq 2$. To extract a basis of $W([n], A)$ it is enough to choose $\#A - 1$ subsets $C_1(A), \dots, C_{\#A - 1}(A)$ of $A$ with size $2$ whose union is the whole $A$. Then a basis is given by the elements $e_{(A, x)} + e_{(A, y)}$ with $x, y$ in $C_i(A)$. Instead $W([n],\{x\})$ is the span of $e_x$, for each $x \in [n]$. In what follows we shall also use the notation $e_{(\{x\},x)}:=e_x$ for each $x \in [n]$.

We define a bracket on $G([n], \bullet)$ in the following manner. Firstly, we declare that the bracket of a graded piece of degree at least $2$ must be $0$, so that we need to define only the bracket between $\text{G}([n],1)=V_{[n]}$ and $\text{G}([n], i)$ for any $i$ in $\mathbb{Z}_{\geq 1}$. For $j \in [n]$ and $A$ containing $j$, we put $[e_j,W([n], A)] =0$. Suppose instead $j \not \in A$ and $A$ not empty. Then there is an operator $T_j$ from $W([n], A)$ to $W([n], A \cup \{j\})$ defined by the following assignment
$$
\sum_{x \in B} e_{(A, x)} \mapsto \sum_{x \in B}e_{(A \cup \{j\},x)}, 
$$
given for each fixed $B \subseteq A$ with $\#B$ even. We put $[e_j,]|_{W([n], A)} = T_j$. Observe that we have a natural projection map
$$
g_{[n]}: \text{G}([n], \bullet) \twoheadrightarrow V_{[n]}.
$$

\begin{theorem} \label{Governing algebra is the universal algebra}
\emph{(a)} The graded vector space $\emph{G}([n], \bullet)$ equipped with $[,]_{[n]}$ is a graded Lie algebra. The pair $((\emph{G}([n], \bullet), [,]_{[n]}), g_{[n]})$ is an $n$-expansion Lie algebra. \\
\emph{(b)} The surjective homomorphism $(g_{[n]})_{\bullet}$ induces an isomorphism
$$
((\emph{C}([n], \bullet), [,]_{[n]}), c_{[n]}) \simeq_{n\emph{-exp. Lie-alg}} ((\emph{G}([n], \bullet), [,]_{[n]}), g_{[n]}).
$$
Every $n$-expansion Lie algebra $((L_{\bullet}, []), \psi)$ admits a natural epimorphism 
$$
\psi_{\bullet}: ((\emph{G}([n], \bullet), [,]_{[n]}), c_{[n]}) \twoheadrightarrow ((L_{\bullet}, []), \psi).
$$
In particular $\emph{dim}_{\mathbb{F}_2}(L_{\bullet}) \leq n2^{n - 1}-2^n+1+n$ and one has equality if and only if $\psi_{\bullet}$ is an isomorphism. 
\end{theorem}

\begin{proof}
Part (a) is a straightforward verification. Part (b) directly follows upon combining Proposition \ref{Consistent tensors are governing tensors} and Proposition \ref{Construction of universal algebra}.
\end{proof}

We remark that the abuse of notation for the bracket $[,]_{[n]}$ and for the map $\psi_{\bullet}$ shall cause no confusion since $((\emph{C}([n], \bullet), [,]_{[n]}), c_{[n]})$ and $((\emph{G}([n], \bullet), [,]_{[n]}), g_{[n]})$ are canonically identified. We call this algebra the $n$-\emph{governing Lie algebra}.

We shall see that Theorem \ref{Governing algebra is the universal algebra} is the crucial tool to establish the analoguous result for $n$-expansion groups. There is a universal $n$-expansion group of size $2^{n2^{n - 1}-2^n+1+n}$. In our reduction to Lie algebras it will be useful to know in advance that any $n$-expansion group is finite with size equal to the size of its Lie algebra. In order to achieve that we firstly prove the following proposition. Recall that if $(G, \phi, (g_1, \dots, g_n))$ is an $n$-expansion group, then axioms $(1),(2)$ and $(3)$ of Definition \ref{Expansion of groups} imply that $[G, G]$ is a module over the ring $\mathbb{F}_2[\frac{G}{[G, G]}]$, and so is any vector sub-space of $[G, G]$ that is normal in $G$. We denote by $I_{\frac{G}{[G, G]}}$ the augmentation ideal of $\mathbb{F}_2[\frac{G}{[G, G]}]$.

\begin{proposition} 
\label{Augmentation=descending central}
Let $(G, \phi, (g_1, \dots, g_n))$ be an $n$-expansion group and let $i \in \mathbb{Z}_{\geq 2}$. Then
$$
I_{\frac{G}{[G, G]}}^{i - 2} \cdot G^{(2)} = G^{(i)}.
$$
\end{proposition}

\begin{proof}
For $i = 2$ this is trivial. We proceed by induction on $i$, so suppose that the statement holds for a given $i \in \mathbb{Z}_{\geq 2}$. By definition we have that $G^{(i + 1)}= [G, G^{(i)}]$. Hence it is spanned by elements of the form $[g_1, g_2]$ with $g_2 \in G^{(i)}$. Note that $[g_1, g_2] = (1 + g_1)g_2$, since $g_2$ is in particular in $[G, G]$. Therefore by the inductive assumption we get that $[g_1, g_2]$ equals $(1 + g_1)v$ where $v$ is a general element of $I_{\frac{G}{[G, G]}}^{i - 2} \cdot [G, G]$. As $g_1$ and $v$ vary this set spans precisely $I_{\frac{G}{[G, G]}}^{i - 1} \cdot [G, G]$.
\end{proof}

\begin{proposition} 
\label{n-expansion groups are finite}
Let $(G, \phi, (g_1, \dots, g_n))$ be an $n$-expansion group. Then 
$$
\#G=\#L_{\bullet}(G) \leq 2^{n2^{n - 1}-2^n+n+1}. 
$$
In particular, $\{g_1, \dots, g_n\}$ is a generating set for $G$.
\end{proposition}

\begin{proof}
Thanks to Proposition \ref{Lie algebra of an expansion is an expansion of a Lie algebra} and Theorem \ref{Governing algebra is the universal algebra} we get that $\#L_{\bullet}(G)$ is indeed finite and satisfies the claimed upper bound. In particular there exists a positive integer $i$ such that $L_i(G)=0$, which is equivalent to $G^{(i)} = G^{(i + 1)}$. Observe that $i > 1$ because $n$ is at least $1$. Hence we can apply Proposition \ref{Augmentation=descending central} twice to conclude that 
$$
I_{\frac{G}{[G, G]}} \cdot G^{(i)} = I_{\frac{G}{[G, G]}}^{i - 1} \cdot G^{(2)} = G^{(i + 1)} = G^{(i)}. 
$$
This implies that for any positive integer $h$ we have 
$$
I_{\frac{G}{[G, G]}}^{h} \cdot G^{(i)} = G^{(i)}. 
$$
On the other hand, by axiom $(1)$ and axiom $(3)$ from Definition \ref{Expansion of groups} we obtain that 
$$
I_{\frac{G}{[G, G]}}^{n+1} = 0.
$$ 
Therefore we conclude that $G^{(i)}$ is the trivial group and $\#G=\#L_{\bullet}(G)$. In particular it follows that $G$ is a finite $2$-group and the set $\{g_1, \dots, g_n\}$ is a generating set, modulo $[G,G]$, by axioms $(1)$ and axioms $(3)$ of Definition \ref{Expansion of groups}. It follows that $\{g_1, \ldots, g_n\}$ generates $G$.
\end{proof}

We next proceed to show that a universal expansion group must exist. Let $F_{[n]}$ be the free group on $n$ letters, which we denote as $x_1, \dots,x_n$. We denote by $N$ the subgroup generated by the squares of all elements in $[F_{[n]},F_{[n]}]$; this is a normal (actually characteristic) subgroup. We denote by $\widetilde{N}$ the largest normal subgroup containing $N$ and $\{x_1^2, \dots, x_n^2\}$. We put
$$
\mathcal{C}([n]) := \frac{F_{[n]}}{\widetilde{N}}.
$$ 
The map $x_i \mapsto e_i$ extends uniquely to a surjective epimorphism $F_{[n]} \twoheadrightarrow V_{[n]}$, which clearly factors through $\widetilde{N}$, yielding a surjective epimorphism
$$
\b{c}_{[n]} : \mathcal{C}([n]) \to V_{[n]}.
$$
We define a homomorphism of $n$-expansion groups 
$$
f: (G, \phi_1, (g_1, \dots, g_n)) \to (G', \phi_2, (g_1', \dots, g_n'))
$$
to be a group homomorphism from $G$ to $G'$ such that $f(g_i) = g'_i$. We observe that thanks to Proposition \ref{n-expansion groups are finite} it follows that among two $n$-expansion groups there is always at most one homomorphism, and if there is one then it must automatically be an epimorphism.

\begin{proposition} 
\label{trivial universal}
$(\mathcal{C}([n]), \b{c}_{[n]}, (x_1, \dots,x_n))$ is an $n$-expansion group. Furthermore, the unique homomorphism $\phi_{\bullet}: F_{[n]} \to G$ sending $x_i$ to $g_i$ yields an epimorphism of $n$-expansion groups
$$
\phi_{\bullet}: (\mathcal{C}([n]), \b{c}_{[n]}, (x_1, \dots,x_n)) \twoheadrightarrow (G, \phi, (g_1, \dots, g_n)).
$$
\end{proposition}

\begin{proof}
Firstly, we verify that $(\mathcal{C}([n]), \b{c}_{[n]}, (x_1, \dots,x_n))$ is an $n$-expansion group. Clearly the classes of $x_i$ are involutions in $\mathcal{C}([n])$ by definition of $\widetilde{N}$. Therefore the group $\mathcal{C}([n])$ is generated by $n$ involutions and therefore any of its abelian quotients will be. It follows that its abelianization is a vector space of dimension at most $n$. On the other hand $\b{c}_{[n]}$ is a surjective homomorphism onto $V_{[n]}$. It follows that $\text{ker}(\b{c}_{[n]})$ must be equal to $[\mathcal{C}([n]), \mathcal{C}([n])]$. 

We claim that $([a_1, b_1] \cdot \ldots \cdot [a_m, b_m])^2$ is the identity if $a_j, b_j \in \mathcal{C}([n])$ for $1 \leq j \leq m$. Indeed, fix for every $j \in [m]$ elements $\widetilde{a}_j, \widetilde{b}_j$ of $F_{[n]}$ that reduce respectively to $a_j$ and $b_j$. Then the product $([\widetilde{a}_1, \widetilde{b}_1] \cdot \ldots \cdot [\widetilde{a}_m, \widetilde{b}_m])^2$ has class in $\mathcal{C}([n])$ equal to
$$
([a_1, b_1] \cdot \ldots \cdot [a_m, b_m])^2,
$$
and at the same time is in $\widetilde{N}$, by definition of $\widetilde{N}$. Therefore the claim holds and we have shown that the triple $(\mathcal{C}([n]), \b{c}_{[n]}, (x_1, \dots,x_n))$ is an $n$-expansion group.

The assignment $x_i \mapsto g_i$ for $i \in [n]$ certainly extends uniquely to a homomorphism from $F_n$ to $G$ thanks to the universal property of a free group. On the other hand since $G$ is an $n$-expansion group, it is clear that the kernel of such a homomorphism must contain $\widetilde{N}$. Hence it induces a homomorphism from $\mathcal{C}([n])$ to $G$. It clearly respects the structure of an $n$-expansion group. Therefore (as we explained above) it must automatically be surjective thanks to Proposition \ref{n-expansion groups are finite}. 
\end{proof}

To state the main results of this section we first introduce the analogues of $\text{G}([n], \bullet)$ for $n$-expansion groups. It will be convenient to introduce for every $i \in [n]$ a different structure of $n$-expansion group for the abstract group
$$
 \mathbb{F}_2[V_{[n - 1]}] \rtimes V_{[n - 1]}
$$
in the following way. We let $V_{[n] - \{i\}}$ act on $V_{\{A \subseteq [n]: i \in A\}}$ by setting
$$
e_j \cdot e_A :=
\left\{
	\begin{array}{ll}
		e_A + e_{A \cup \{j\}}  & \mbox{if } j \not \in A \\
		e_A & \mbox{if } j \in A.
	\end{array}
\right.
$$
Define
$$
G_i([n]) := V_{\{A \subseteq [n]: i \in A\}} \rtimes V_{[n] - \{i\}}.
$$
Define $g_{i, j} := (\text{id}, e_j)$ for $j \in [n] - \{i\}$ and $g_{i, i} := (\{i\}, \text{id})$. The group $G_i([n])$ has a natural projection map $\phi_i([n])$ onto $V_{[n]}$, which sends $g_{i, j}$ to $e_j$.

\begin{proposition} 
\label{single expansion group}
For every $i \in [n]$ we have that $(G_i([n]), \phi_i([n]), (g_{i, 1}, \dots, g_{i, n}))$ is an $n$-expansion group. 
\end{proposition}

\begin{proof}
Recall that if $G$ is a group and $A$ is a $G$-module, then $[A \rtimes G, A \rtimes G] = I_G \cdot A \rtimes G^{(2)}$. This general fact gives us readily that axiom $(1)$ and axiom $(3)$ are respected. Axioms $(2)$ and $(4)$ are clearly satisfied. 
\end{proof}

We next introduce a general construction (a similar construction could have been introduced for Lie algebras).
 
\begin{definition} 
\label{Fibered products}
Let $(G, \phi, (g_1, \dots, g_n)), (G', \phi', (g'_1, \dots, g'_n))$ be $n$-expansion groups. We denote by $(G, \phi, (g_1, \dots, g_n)) \times (G', \phi', (g'_1, \dots, g'_n))$ the triple $(G'', \phi'', (g''_1, \dots, g''_n))$ defined in the following manner. The symbol $G''$ stands for the subgroup of $G \times G'$ generated by the pairs $(g_i, g'_i)$. This is a subgroup of the fibered product over $V_{[n]}$ (with respect to $\phi, \phi'$). Then $\phi''$ denotes the restriction to $G''$ of the natural map from the fibered product to $V_{[n]}$. Finally $g''_i$ denotes the element $(g_i, g_i')$ of $G''$. 
\end{definition}

\begin{proposition} 
\label{Fibered products exists}
Let $(G, \phi, (g_1, \dots, g_n))$ and $(G', \phi', (g'_1, \dots, g'_n))$ be two $n$-expansion groups. Then $(G, \phi, (g_1, \dots, g_n)) \times (G', \phi', (g'_1, \dots, g'_n))$  is an $n$-expansion group.
\end{proposition}

\begin{proof}
Axiom $(1)$ and axiom $(4)$ of Definition \ref{Expansion of groups} hold because the elements $g''_i = (g_i, g'_i)$ are coordinatewise involutions mapping to $e_i$. Axiom $(2)$ holds because the commutator subgroup of $G''$ will certainly be a subgroup of $[G, G] \times [G', G']$. But by definition of an $n$-expansion both of them are vector spaces over $\mathbb{F}_2$ and therefore so is any subgroup of it. Hence axioms $(1),(2)$ and $(4)$ have been established. 

We now check axiom $(3)$. We observe that the group $G''$ is, by construction, generated by $n$ involutions, namely by the set $\{g''_i\}_{i \in [n]}$, therefore every abelian quotient of it is also generated by $n$ involutions and therefore must be a vector space over $\mathbb{F}_2$ of dimension at most $n$. Since, as we have shown, axiom $(1)$ holds and $\phi''$ is surjective, it follows that $\text{ker}(\phi'')$ must be the commutator subgroup of $G''$, since the quotient by this group is a vector space of dimension precisely equal to $n$. 
\end{proof}

Let us notice that the above product of $n$-expansion groups does also satisfy the categorical property of a product. Namely, if an $n$-expansion group surjects (as expansion group)  on both factors (with then a necessarily unique map of $n$-expansion groups) then it surjects (as expansion group) on the product (with then a necessarily unique map of $n$-expansion groups that necessarily commutes with the projections on the two factors, again by uniqueness). 

We now define 
$$
(\mathcal{G}([n]), \b{g}_{[n]}, (e_1, \dots, e_n)) := \prod_{i \in [n]} (G_i([n]), \phi_i([n]), (g_{i, 1}, \dots, g_{i, n})). 
$$
Thanks to Proposition \ref{single expansion group} and Proposition \ref{Fibered products exists} this is an $n$-expansion group. 
It is a simple verification that the Lie algebra attached to $\prod_{i \in [n]} (G_i([n]), \phi_i([n]),(g_{i, 1}, \dots, g_{i, n}))$ is precisely $((\emph{G}([n], \bullet), [,]_{[n]}), c_{[n]})$. One can see this quite easily by a direct calculation with nested commutators.

Alternatively, one can set up fibered product also for $n$-expansion Lie algebras and find that 
$((\emph{G}([n], \bullet), [,]_{[n]}), c_{[n]})$ is the fibered product of the Lie algebras of the various  $(G_i([n]), \phi_i([n]),(g_{i, 1}, \dots, g_{i, n}))$. Then one concludes by using the general fact that the Lie algebra of a product of two $n$-expansion groups is the product of the Lie algebras, where each product is taken in the respective categories.

Recall that Proposition \ref{trivial universal} provides for any $n$-expansion group the unique epimorphism of $n$-expansion groups $\phi_{\bullet}: (\mathcal{C}([n]), \b{c}_{[n]}, (x_1, \dots,x_n)) \twoheadrightarrow (G, \phi, (g_1, \dots, g_n))$.

\begin{theorem} 
\label{governing group is universal}
$(\mathcal{G}([n]), \b{g}_{[n]}, (e_1, \dots, e_n))$ is an $n$-expansion group and $(\b{g}_{[n]})_{\bullet}$ induces an isomorphism of $n$-expansion groups. Every $n$-expansion group $(G, \phi, (g_1, \dots, g_n))$ admits a unique epimorphism of $n$-expansion groups
$$
\phi_{\bullet} \circ (\b{g}_{[n]})_{\bullet}^{-1}: (\mathcal{G}([n]), \b{g}_{[n]}, (e_1, \dots, e_n)) \twoheadrightarrow (G, \phi, (g_1, \dots, g_n)).
$$
In particular, $\#G \leq 2^{n2^{n - 1}-2^n+n+1}$ with equality if and only if the map $\phi_{\bullet} \circ (\b{g}_{[n]})_{\bullet}^{-1}$ is an isomorphism. 
\end{theorem}

\begin{proof}
The first part follows immediately from Theorem \ref{Governing algebra is the universal algebra} and  Proposition \ref{n-expansion groups are finite}. Hence the second part follows from Proposition \ref{trivial universal}. Finally, the third part follows from the first part, Proposition \ref{counting} and Proposition \ref{n-expansion groups are finite}.
\end{proof}

\begin{remark} 
\label{More on Fibered products} 
The existence of a universal object, i.e. $\mathcal{G}([n])$ or the isomorphic $\mathcal{C}([n])$, follows, for a purely formal reason, already from the existence of a universal bound on the cardinality as for instance the one given in Proposition \ref{n-expansion groups are finite}. Indeed, such a bound implies that there are only finitely many isomorphism classes of $n$-expansion groups and thus the product (as $n$-expansion groups) of the finitely many isomorphism classes will be the sought universal object. 

Similar comments apply in the case of an $n$-expansion Lie algebra and the algebra $C([n], \bullet)$. This gives a slightly alternative approach to the \emph{a priori} proof that the universal objects exists, i.e. of Proposition \ref{trivial universal}. Note that the existence of a universal object is a vital input for the proof of Theorem \ref{governing group is universal}, namely when one applies Proposition \ref{trivial universal}. 

This point of view shows that $\mathcal{G}([n])$ is naturally defined as a product of a few explicit $n$-expansion groups. The meaning of Theorem \ref{governing group is universal} is then that from the trivial universal object, gotten by multiplying out all $n$-expansions, one can extract a simple explicit presentation, namely the definition of $\mathcal{G}([n])$.
\end{remark}

\subsection{\texorpdfstring{$[(k_1, \dots, k_n)]$-expansion groups and $[(k_1, \dots, k_n)]$-expansion Lie algebras}{[(k1, ..., kn)]-expansion groups and [(k1, ..., kn)]-expansion Lie algebras}}
Let $(k_1, \dots, k_n) \in \mathbb{Z}_{\geq 1}^{n}$. We systematically identified the spaces $V_{[(k_1, \dots, k_n)]}$ and $V_{[k_1+ \dots +k_n]}$ by concatenation of coordinates in Section \ref{Auxiliary tensor spaces}, and we shall continue to do so. Let us now define a $[(k_1, \dots, k_n)]$-expansion group. 

\begin{definition} 
\label{definition of (k1, ... ,kn)-expansion groups}
We call a $k_1 + \ldots + k_n$-expansion group $(G, \phi,(g_1, \dots, g_{k_1+ \dots + k_n}))$ a $[(k_1, \dots, k_n)]$-expansion group in case $\phi^{-1}(\text{ker}(\pi_{(k_1, \dots, k_n)}))$ is a $\mathbb{F}_2$-vector space. 
\end{definition}

Observe that a $(1, \dots, 1)$-expansion group, i.e. $k_i = 1$ for each $i \in [n]$, is precisely an $n$-expansion group. We similarly define a $(k_1, \dots, k_n)$-expansion Lie algebra. 

\begin{definition} 
\label{definition of (k1, ... ,kn)-expansion Lie algebra}
We call a $k_1+ \dots + k_n$-expansion Lie algebra $(L_{\bullet}, \psi)$ a $[(k_1, \dots, k_n)]$-expansion Lie algebra in case the following two conditions are satisfied:\\
$(1)$ The bracket restricted to the space $\psi^{-1}(\text{ker}(\pi_{(k_1, \dots, k_n)}))$ is trivial. In other words, $\psi^{-1}(\text{ker}(\pi_{(k_1, \dots, k_n)}))$ is an abelian sub-algebra of $L_{\bullet}$. \\
$(2)$ Let $j$ be in $\mathbb{Z}_{\geq 2}$ and let $(i_1, \dots, i_j)$ be a vector in $[k_1+ \dots + k_n]^j$ with two distinct entries belonging to the same $(k_1, \dots, k_n)$-block. Then the nested commutator 
$$
[\psi|_{L_1}^{-1}(e_{i_1}), [\psi|_{L_1}^{-1}(e_{i_2}), [\dots, [\psi|_{L_1}^{-1}(e_{i_{j - 1}}), \psi|_{L_1}^{-1}(e_{i_j})] \dots ]]].
$$
vanishes. 
\end{definition}

The following fact explains the connection between Definition \ref{definition of (k1, ... ,kn)-expansion groups} and Definition \ref{definition of (k1, ... ,kn)-expansion Lie algebra}. 

\begin{proposition} 
\label{Lie algebra of (k1, ... ,kn)-expansion group is (k1, ... ,kn)-expansion Lie algebra}
Let $(G, \phi,(g_1, \dots, g_{k_1+ \dots + k_n}))$ be a $[(k_1, \dots, k_n)]$-expansion group. Then $((L_{\bullet}(G), [,]_{G}), L_{\bullet}(\phi))$ is a $[(k_1, \dots, k_n)]$-expansion Lie algebra. 
\end{proposition}

\begin{proof}
We deduce from Proposition \ref{Lie algebra of an expansion is an expansion of a Lie algebra} that $((L_{\bullet}(G), [,]_{G}), L_{\bullet}(\phi))$ is a $k_1+ \dots + k_n$-expansion Lie algebra. The Lie algebra $((L_{\bullet}(G), [,]_{G}), L_{\bullet}(\phi))$ also satisfies axiom $(1)$ from Definition \ref{definition of (k1, ... ,kn)-expansion Lie algebra}. Indeed, this follows from our assumption that $\phi^{-1}(\text{ker}(\pi_{(k_1, \dots, k_n)}))$ is abelian.

Since $\phi^{-1}(\text{ker}(\pi_{(k_1, \dots, k_n)}))$ is a $\mathbb{F}_2$-vector space, we have for all $h_1, h_2$ in the same $[(k_1, \dots, k_n)]$-block of $k_1+ \dots + k_n$ that $g_{h_1}g_{h_2}$ is an involution. But since $g_{h_1}, g_{h_2}$ are themselves involutions, we get the identity
$$
[g_{h_1}, g_{h_2}] = (g_{h_1}g_{h_2})^2 = \text{id}.
$$ 
This already shows axiom $(2)$ for $j = 2$. Now let us consider the case $j \geq 3$. We shall distinguish three cases. First suppose that the repetition happens in the first $j - 2$ blocks. Then observe that 
$$
[\psi|_{L_1}^{-1}(e_{i_1}), [\psi|_{L_1}^{-1}(e_{i_2}), [\dots, [\psi|_{L_1}^{-1}(e_{i_{j - 1}}), \psi|_{L_1}^{-1}(e_{i_j})] \dots ]]] = \prod_{1 \leq h \leq j - 2}(1+g_{i_h})[g_{i_{j - 1}}, g_{i_j}],
$$
where, thanks to the fact that $\phi^{-1}(\text{ker}(\pi_{(k_1, \dots, k_n)}))$ is abelian, the action factors completely through $\phi^{-1}(\text{ker}(\pi_{(k_1, \dots, k_n)}))$. Therefore such a repetition would imply that we get the square of $1 + g_{i_h}$ for some $1 \leq h \leq j - 2$ and thus we get $0$. 

Since we already settled the case $j = 2$, we are certainly done if the repetition occurs in the last two indices. Therefore we are left with the case that one index is smaller than $j - 1$ and the other one is in the set $\{j - 1, j\}$. By symmetry of the bracket (we are over $\mathbb{F}_2$) we may assume that the index in $\{j - 1, j\}$ is actually $j - 1$. Furthermore, since $((L_{\bullet}(G), [,]_{G}), L_{\bullet}(\phi))$ is a $k_1+ \dots + k_n$-expansion Lie algebra, we may assume by symmetry in the first $j - 2$ entries of nested commutators coming from axiom $(4)$ that the index before $j - 1$ is $j - 2$. 

Hence we have to analyze the commutator 
$$
[g_{i_{j-2}}, [g_{i_{j - 1}}, g_{i_j}]]
$$
under the condition that $i_{j - 2}$ and $i_{j - 1}$ are in the same block. We claim that 
$$
[g_{i_{j-2}}, [g_{i_{j - 1}}, g_{i_j}]] = [g_{i_{j - 1}}, [g_{i_{j - 1}}, g_{i_j}]]. 
$$
Indeed, because $i_{j - 2}$ and $i_{j - 1}$ are in the same block, it follows that $g_{i_{j-2}}g_{i_{j - 1}}$ is an element of $\phi^{-1}(\text{ker}(\pi_{(k_1, \dots, k_n)}))$, which therefore commutes with $\phi^{-1}(\text{ker}(\pi_{(k_1, \dots, k_n)}))$. In particular,  $g_{i_{j-2}}g_{i_{j - 1}}$ commutes with any commutator, so it certainly commutes with $[g_{i_{j - 1}}, g_{i_j}]$. Therefore 
$$
[g_{i_{j-2}}, [g_{i_{j - 1}}, g_{i_j}]] = [g_{i_{j-2}}(g_{i_{j-2}}g_{i_{j - 1}}), [g_{i_{j - 1}}, g_{i_j}]] = [g_{i_{j - 1}}, [g_{i_{j - 1}}, g_{i_j}]]
$$
as we claimed. The last term vanishes by the same computation following equation (\ref{etau}), which completes the proof.
\end{proof}

Let now $((L_{\bullet}, []), \psi)$ be a $[(k_1, \dots, k_n)]$-expansion Lie algebra. From Proposition \ref{expansion bracket give consistent tensors} we get for every $i \in \mathbb{Z}_{\geq 1}$ a natural multi-linear map
$$
\psi_i: (V_{[(k_1, \dots, k_n)]})^{\otimes i} \to L_i
$$
by taking nested commutators of length $i$. The dual of $\psi_i$ satisfies 
$$
\text{Im}(\psi_i^{\vee}) \subseteq \text{Cons}(V_{[(k_1, \dots, k_n)]}, i).
$$
So far we have only used that $((L_{\bullet}, [, ]), \psi)$ is a $k_1+ \dots + k_n$-expansion Lie algebra. However we see that axiom $(1)$ and $(2)$ of Definition \ref{definition of (k1, ... ,kn)-expansion Lie algebra} immediately give us the following additional constraint.

\begin{proposition} 
\label{expansion bracket give tilde consistent tensors}
We have that
$$
\emph{Im}(\psi_i^{\vee}) \subseteq \widetilde{\emph{Cons}}(V_{[(k_1, \dots, k_n)]}, i).
$$
\end{proposition}

\noindent The inclusions $\widetilde{\text{Cons}}(V_{[(k_1, \dots, k_n)]}, i) \subseteq \text{Cons}(V_{[(k_1, \dots, k_n)]}, i)$ for each $i \in \mathbb{Z}_{\geq 1}$ naturally give rise to surjections
$$
\text{Cons}(V_{[(k_1, \dots, k_n)]}, i)^{\vee} \twoheadrightarrow \widetilde{\text{Cons}}(V_{[(k_1, \dots, k_n)]}, i)^{\vee}.
$$
When we bunch these surjections together, we get an epimorphism of $(k_1+ \dots + k_n)$-expansion Lie algebras
$$
((C([k_1+ \dots + k_n], \bullet), [,]_{[k_1+ \dots + k_n]}), c_{[k_1+ \dots + k_n]}) \twoheadrightarrow ((\widetilde{C}([(k_1, \dots, k_n)], \bullet), [,]_{[n]}), \widetilde{c}_{[(k_1, \dots, k_n)]}).
$$
By construction the algebra on the right hand side is the universal $[(k_1, \dots, k_n)]$-expansion Lie algebra; it is a quotient of the universal $k_1 + \ldots + k_n$-expansion Lie algebra. Thanks to Proposition \ref{counting more} and Proposition \ref{Consistent tensors are governing tensors more general} we have
$$
\# \widetilde{C}([(k_1, \dots, k_n)],\bullet) = 2^{(k_1+ \dots + k_n) \cdot 2^{n - 1} - 2^n + 1 + n}.
$$
Next, in the same spirit of Proposition \ref{Governing algebra is the universal algebra}, one can refine this counting into a presentation of the universal algebra in terms of governing tensors. This can simply be done by dualizing the inclusions
$$
\widetilde{\text{Gov}}(V_{[(k_1, \dots, k_n)]}, i) \subseteq \text{Gov}(V_{[(k_1, \dots, k_n)]}, i).
$$
Bunching together the natural surjections then yields an epimorphism
$$
((\text{G}([k_1+ \dots + k_n], \bullet), [,]_{[n]}), g_{[k_1+ \dots + k_n]}) \twoheadrightarrow ((\widetilde{\text{G}}([n], \bullet), [,]_{[k_1+ \dots + k_n]}), \widetilde{g}_{[k_1+ \dots + k_n]}), 
$$
which gives an explicit presentation of the universal $[(k_1, \dots, k_n)]$-expansion Lie algebra in terms of governing tensors. 

Now we consider the case of $[(k_1, \dots, k_n)]$-expansion groups. Since these are in particular $(k_1+ \dots + k_n)$-expansion groups, we know, thanks to Proposition \ref{n-expansion groups are finite} that they are finite with size equal to the size of their Lie algebra. Next, thanks to Proposition \ref{expansion bracket give tilde consistent tensors}, we conclude that the size of each such group is at most $2^m$, where $m$ equals $(k_1+ \dots + k_n) \cdot 2^{n - 1}-2^n+1+n$.

Now, let $i \in [n]$ and let $x$ be any element of the $i$-th block of $k_1+ \dots + k_n$, i.e. an element of $\{\sum_{1 \leq j \leq i - 1}k_i + 1, \dots, \sum_{1 \leq j \leq i}k_i\}$. Then we have a unique surjection of abstract groups
$$
\phi_{x}: \mathcal{G}([k_1+ \dots + k_n]) \twoheadrightarrow \mathbb{F}_2[\mathbb{F}_2^{[n]-\{i\}}] \rtimes \mathbb{F}_2^{[n]-\{i\}}
$$
that extends the following assignment. For $j \in [k_1+ \dots + k_n]$ in the $h$-th block, with $h \neq i$, we send $g_j$ to $e_h$, where the notation $e_h$ follows our conventions for
$$
\{0\} \rtimes \mathbb{F}_2^{[n]-\{i\}} \subseteq \mathbb{F}_2[\mathbb{F}_2^{[n] - \{i\}}] \rtimes \mathbb{F}_2^{[n] - \{i\}}. 
$$
Instead, the elements of the $i$-th block are sent to the $1$-dimensional subspace 
$$
\mathbb{F}_2 \cdot \{(0, \dots, 0)\} \subseteq \mathbb{F}_2[\mathbb{F}_2^{[n] - \{i\}}]
$$ 
through the character $\chi_x$. The intersection of $\text{ker}(\phi_x)$ as $x$ varies in $[k_1+ \dots + k_n]$ gives a quotient of $\mathcal{G}([k_1+ \dots + k_n])$ that is a $[(k_1, \dots, k_n)]$-expansion group by using the same generators as $\mathcal{G}([k_1+ \dots + k_n])$ and the same projection map to $\mathbb{F}_2^{[k_1+ \dots + k_n]}$. We denote this object by $\widetilde{\mathcal{G}}([k_1+ \dots + k_n])$. 

A direct verification with nested commutators shows that it has size precisely equal to $2^m$, where $m$ is $(k_1+ \dots + k_n) \cdot 2^{n - 1}-2^n+1+n$. At this stage there are several ways to conclude this. A neat argument is given at the end of Section \ref{expansion maps as coordinates of monomials}. Therefore it must be the universal $[(k_1, \dots, k_n)]$-expansion group. We summarize this analysis in the following.

\begin{theorem} 
\label{governing group is universal tilde}
The triple
$$
(\widetilde{\mathcal{G}}([k_1+ \dots + k_n]), \b{g}_{[k_1+ \dots + k_n]}, (e_1, \dots, e_{k_1+ \dots + k_n}))
$$ 
is a $[(k_1, \dots, k_n)]$-expansion group. If $(G, \phi, (g_1, \dots, g_{k_1+ \dots + k_n}))$ is any $[(k_1, \dots, k_n)]$-expansion group, then there is a unique epimorphism of $[(k_1, \dots, k_n)]$-expansion groups
$$ 
(\widetilde{\mathcal{G}}([k_1+ \dots + k_n]), \b{g}_{[n]}, (e_1, \dots, e_{k_1+ \dots + k_n})) \twoheadrightarrow (G, \phi, (g_1, \dots, g_{k_1+ \dots + k_n})).
$$
In particular, $\#G \leq 2^{(k_1+ \dots + k_n) \cdot 2^{n - 1}-2^n+n+1}$ and one has equality if and only if the above map is an isomorphism.  
\end{theorem}

\subsection{Intermezzo: expansion maps as coordinates of monomials} \label{expansion maps as coordinates of monomials} 
Expansion maps were introduced by Smith in \cite[eq. (2.2)]{smith2}. In this section we provide an alternative definition of an expansion map, one that will make clear the connection between expansion maps and expansion groups. A simple computation will show that the present definition is equivalent to the one in Smith.

Let $m \in \mathbb{Z}_{\geq 1}$. Let $G$ be a profinite group and let $A$ be a linearly independent finite set of continuous characters from $G$ to $\mathbb{F}_2$, containing at least two elements. Let $\chi_0$ be in $A$. 

\begin{definition} 
\label{def of expansion}
An expansion map for $G$, with \emph{support} $A$ and \emph{pointer} $\chi_0$, is a continuous group homomorphism
$$
\psi:G \to \mathbb{F}_2[ \mathbb{F}_2^{A - \{\chi_0\}}] \rtimes \mathbb{F}_2^{A - \{\chi_0\}}
$$
such that for every $\chi \in A - \{\chi_0\}$ we have that the natural projection on the $\chi$-th coordinate of $\mathbb{F}_2[ \mathbb{F}_2^{A - \{\chi_0\}}] \rtimes \mathbb{F}_2^{A - \{\chi_0\}}$ composed with $\psi$ equals $\chi$; while the unique non-trivial character $\mathbb{F}_2[ \mathbb{F}_2^{A - \{\chi_0\}}] \rtimes \mathbb{F}_2^{A - \{\chi_0\}} \to \mathbb{F}_2$ that sends the subgroup $\{0\} \rtimes \mathbb{F}_2^{A - \{\chi_0\}}$ to $\{0\}$ composed with $\psi$ equals $\chi_0$. 
\end{definition}

Observe that $\mathbb{F}_2[ \mathbb{F}_2^{A - \{\chi_0\}}]$ is a polynomial ring with a basis given by the set of square-free monomials $t_B := \prod_{j \in B}t_j$ for any $B \subseteq A - \{\chi_0\}$. Hence, projection on monomials gives a collection of functions
$$
\phi_B(\psi): G \to \mathbb{F}_2
$$
that allows us to reconstruct $\psi$ via the simple formula
$$
\psi(g) = \left(\sum_{B \subseteq A - \{\chi_0\}} \phi_B(\psi)(g)t_B, (\chi(g))_{\chi \in A - \{\chi_0\}}\right).
$$
Define $\chi_S := \prod_{\chi \in S} \chi$. Using the composition law in a semi-direct product and expanding the product we obtain the equation
\begin{align}
\label{eSmith}
(d\phi_B(\psi))(g_1, g_2) &:= \phi_B(\psi)(g_1g_2) + \phi_B(g_1) + \phi_B(g_2) \nonumber \\
&= \sum_{\emptyset \subsetneq S \subseteq B} \chi_{S}(g_1) \phi_{B - S}(\psi)(g_2),
\end{align}
and conversely any collection of maps $\{\phi_B\}_{B \subseteq A - \{\chi_0\}}$ satisfying equation (\ref{eSmith}) and $\phi_{\emptyset} = \chi_0$ gives rise to an expansion map with support $A$ and pointer $\chi_0$. Indeed, we can define our expansion map by
$$
g \mapsto \left(\sum_{B \subseteq A - \{\chi_0\}} \phi_B(\psi)(g)t_B, (\chi(g))_{\chi \in A - \{\chi_0\}}\right).
$$
This shows that Definition \ref{def of expansion} and Definition \cite[eq. (2.2)]{smith2} are equivalent. We will now prove an important lemma for the coming section.

\begin{lemma}
\label{lComm}
Let $G$ be a profinite group and let $\psi$ be an expansion map for $G$ with support $A$ and pointer $\chi_0$. Then we have for all $\sigma_1, \ldots, \sigma_{\#A} \in G$
$$
\phi_A(\psi)([\sigma_{1}, [\sigma_{2}, [\dots, [\sigma_{\#A - 1}, \sigma_{\#A}] \dots ]]]]) = \phi_{(A, \chi_0)}(\sigma_1, \dots, \sigma_{\#A}),
$$
where we define
\[
\phi_{(A, \chi_0)} := \sum_{\substack{f \in \textup{Isom}_{\textup{Set}}([\#A], A) \\ f(\#A - 1) = \chi_0 \ \textup{or} \ f(\#A) = \chi_0}} f(1) \otimes \dots \otimes f(\#A).
\]
\end{lemma}

\begin{proof}
Since the commutator group of $\mathbb{F}_2[ \mathbb{F}_2^{A - \{\chi_0\}}] \rtimes \mathbb{F}_2^{A - \{\chi_0\}}$ is a vector space over $\mathbb{F}_2$, we immediately see that
\begin{multline*}
\psi([\sigma_{1}, [\sigma_{2}, [\dots, [\sigma_{\#A - 1}, \sigma_{\#A}] \dots ]]]]) = \\
\left(\prod_{i \in [\#A-2]} \left(\sum_{\chi \in A-\{\chi_0\}} \chi(\sigma_i)t_{\{\chi\}} \right)\right) \cdot [\psi(\sigma_{\#A - 1}), \psi(\sigma_{\#A})].
\end{multline*}
Observe that the element $[\psi(\sigma_{\#A - 1}), \psi(\sigma_{\#A})]$ will admit an expansion of the form
$$
\sum_{\emptyset \neq B \subseteq A-\{\chi_0\}}\lambda_{B}t_B,
$$
as any commutator of the group $\mathbb{F}_2[ \mathbb{F}_2^{A - \{\chi_0\}}] \rtimes \mathbb{F}_2^{A - \{\chi_0\}}$ does. Clearly the operator 
$$
\prod_{i \in [\#A-2]} \left(\sum_{\chi \in A-\{\chi_0\}} \chi(\sigma_i)t_{\{\chi\}} \right)
$$ 
will kill all monomials of degree at least $2$. On the other hand, for each $\chi \in A-\{\chi_0\}$, we have that
$$
\lambda_{\{\chi\}}=\chi(\sigma_{\#A-1})\chi_0(\sigma_{\#A})+\chi(\sigma_{\#A})\chi_0(\sigma_{\#A-1}).
$$
Therefore, expanding the product 
$$
\left(\prod_{i \in [\#A-2]} \hspace{-0.1cm} \left(\sum_{\chi \in A-\{\chi_0\}}\chi(\sigma_i)t_{\{\chi\}} \right)\right) \cdot \hspace{-0.2cm} \sum_{\chi \in A-\{\chi_0\}}(\chi(\sigma_{\#A-1})\chi_0(\sigma_{\#A})+\chi(\sigma_{\#A})\chi_0(\sigma_{\#A-1}))t_{\{\chi\}},
$$
using ordinary multiplication of polynomials, and recalling that the square of each variable is $0$, yields the desired result. 
\end{proof}

We finish this section with a simple argument that the group $\widetilde{\mathcal{G}}([k_1+ \dots + k_n])$ has size $2^m$, where $m=(k_1+ \dots + k_n) \cdot 2^{n-1}-2^n+1+n$. We claim that the collection of nested commutators $[g_{r_1}, [g_{r_2}, [\dots, [g_{r_{i - 1}}, g_{r_i}] \dots ]]]]$, with $2 \leq i \leq n$ and $(r_1, \dots, r_i)$ satisfying requirements $(a)$ to $(e)$ as in the proof of Proposition \ref{Consistent tensors are governing tensors more general}, form a linearly independent set of vectors in $[\widetilde{\mathcal{G}}([k_1+ \dots + k_n]),\widetilde{\mathcal{G}}([k_1+ \dots + k_n])]$. Note that this is indeed a vector space over $\mathbb{F}_2$ by definition of a $k_1+ \dots + k_n$-expansion. Clearly, the claim implies the desired conclusion by counting, with the same argument as in Proposition \ref{counting more}. 

To see the claim, we simply argue as follows. By construction, for each $(r_1, \dots, r_i)$ as above, we can construct an expansion for $\widetilde{\mathcal{G}}([k_1+ \dots + k_n])$, with pointer $\chi_{r_i}$ if $r_i$ is not in the image of $g$ and with pointer $\chi_{r_{i-1}}$ if $r_i$ is in the image of $g$. From this, we get a corresponding map $\phi_{(r_1, \dots, r_i)}:\widetilde{\mathcal{G}}([k_1+ \dots + k_n]) \to \mathbb{F}_2$. These functions are all quadratic characters on $[\widetilde{\mathcal{G}}([k_1+ \dots + k_n]),\widetilde{\mathcal{G}}([k_1+ \dots + k_n])]$ and thanks to Proposition \ref{lComm} we see that it takes value $1$ on the nested commutator $[g_{r_1}, [g_{r_2}, [\dots, [g_{r_{i - 1}}, g_{r_i}] \dots ]]]]$. 

Furthermore, Proposition \ref{lComm} also implies that it vanishes at all other commutators of length $i$. It certainly vanishes at nested commutators of length larger than $i$, because it factors through a group where nested commutators of length $i+1$ vanish. One can also show that, by construction, it vanishes on the nested commutators of length smaller than $i$ among the ones listed above (making use of the fact that they are among the privileged generators). 

\subsection{Reconstructing an expansion group from its corners} 
\label{Reconstructing from corners}
Let $n$ be a positive integer and let $(k_1, \dots, k_n) \in \mathbb{Z}_{\geq 1}^n$. Let $(G, \phi,(g_1, \dots, g_{k_1+ \dots + k_n}))$ be a $[(k_1, \dots, k_n)]$-expansion group. For $i \in [n]$, let $N_i$ be the normal subgroup of $G$ generated by all the $g_j$ with $j$ in the $i$-th block. We have in this way a natural structure of $[(k_h)_{h \neq i}]$-expansion group on $G/N_i$ by using the remaining generators and composing the map $\phi$ with the projection on the coordinates outside the $i$-th block. 

We call this $[(k_h)_{h \neq i}]$-expansion group the $i$\emph{-th corner} of $(G, \phi,(g_1, \dots, g_{k_1+ \dots + k_n}))$ and denote it with $(G, \phi,(g_1, \dots, g_{k_1+ \dots + k_n}))_i$. The goal of this section is to examine which extra  data permits to reconstruct $(G, \phi,(g_1, \dots, g_{k_1+ \dots + k_n}))$ from the collection of its corners $\{(G, \phi,(g_1, \dots, g_{k_1+ \dots + k_n}))_i\}_{i \in [n]}$. As we shall see, the answer is roughly: an explicit collection of $2$-cocycles obtained from linear combinations of expansion maps as introduced in Section \ref{expansion maps as coordinates of monomials}. 

We remind the reader that we defined a function $f$ in equation (\ref{ef}). We begin by defining a space of functions
$$
\Phi[(k_1, \dots, k_n)] \subseteq \text{Map}(\widetilde{\mathcal{G}}([k_1, \dots, k_n]), \mathbb{F}_2)
$$
in the following manner. Thanks to the material of Section \ref{expansion maps as coordinates of monomials} and the definition of $\widetilde{\mathcal{G}}([k_1, \dots, k_n])$ we see that we naturally obtain a map for each $A \subseteq [n]$, $i \in A$ and $x \in f(i)$
$$
\Phi_{(A, x)}: \widetilde{\mathcal{G}}([k_1, \dots, k_n]) \to \mathbb{F}_2,
$$
which furthermore satisfies the equation
$$
(d\Phi_{(A, x)})(\sigma, \tau)=\sum_{\emptyset \neq B \subseteq A - \{i\}}\chi_{B}(\sigma)\Phi_{(A - B,x)}(\tau).
$$
Here $\chi_{B}$ equals $\prod_{j \in B} \chi_{j}$ for $B \neq \emptyset$, where $\chi_{j}$ is the character obtained by summing all projections in the $j$-th block. We also put $\chi_{\emptyset}:=0$. Observe that in case $|A|=1$, the map $\Phi_{(A, x)}$ coincides with the character $\chi_{x}$. 

\begin{proposition}
\label{Shuffling relation}
For each $A \subseteq [n]$ we have that 
$$
\sum_{i \in A} \sum_{x \in f(i)} \Phi_{(A, x)} = \chi_{A}. 
$$
\end{proposition}

\begin{proof}
We shall proceed by strong induction on the size of $A$. If $\#A \leq 1$ the statement is a complete triviality. We assume henceforth that $\#A \geq 2$ and that the claimed relation holds for any $B \subseteq [n]$ with $\#B<\#A$. Applying $d$ to the left hand side we get
$$
\sum_{\emptyset \neq B \subsetneq A} \chi_B(\sigma) \left(\sum_{i \in A-B} \sum_{x \in f(i)} \Phi_{(A - B, x)}(\tau)\right) = \sum_{\emptyset \neq B \subsetneq A} \chi_B(\sigma) \chi_{A-B}(\tau) = (d\chi_A)(\sigma, \tau),
$$ 
where we use the inductive assumption in the first equality and a simple calculation in the second equality.

We conclude that $\sum_{i \in A} \sum_{x \in f(i)} \Phi_{(A, x)}$ and $\chi_{A}$ differ by a character of  $\widetilde{\mathcal{G}}([k_1, \dots, k_n])$. However, since $|A| \geq 2$, we have that each $\sum_{x \in f(i)} \Phi_{(A, x)}$ vanishes on the distinguished generators of the $[k_1, \dots, k_n]$-expansion group $\widetilde{\mathcal{G}}([k_1, \dots, k_n])$. The same applies to $\chi_{A}$, again because $|A| \geq 2$. We conclude that $\sum_{i \in A} \sum_{x \in f(i)} \Phi_{(A, x)}$ and $\chi_{A}$ differ by the trivial character, which is precisely the desired conclusion. 
\end{proof}

When $A$ changes, the corresponding set of functions generate spaces in direct sum by Lemma \ref{lComm}. Altogether we see that the colllection $\{\Phi_{(A, x)}\}_{A \subseteq [n], \#A \geq 2}$ with $x$ varying in a block inside $A$, is a collection of linearly independent functions. We define $\Phi[(k_1, \dots, k_n)]$ to be the space they generate (as a basis) together with all the characters of $\widetilde{\mathcal{G}}([k_1, \dots, k_n])$. 

For $i \in [n]$ we have a natural \emph{cornering} operator $P_i$. Namely for $i \in [n]$, we set
$$
P_i(\Phi_{(A, x)}) :=
\left\{
\begin{array}{ll}
\Phi_{(A - \{i\},x)}  & \mbox{if } i \in A \text{ and } x \not \in f(i)\\
0 & \mbox{otherwise.}
\end{array}
\right.
$$
We put the value of $P_i$ on a character to always be $0$. In this way the assignment extends to a surjective linear operator
$$
P_i: \Phi[(k_1, \dots, k_n)] \twoheadrightarrow \Phi[(k_h)_{h \neq i}].
$$
One obtains after two applications of Proposition \ref{Shuffling relation} that for any $i \in [n]$
$$
P_i(\chi_{A}) :=
\left\{
	\begin{array}{ll}
		\chi_{A-\{i\}}  & \mbox{if } i \in A \\
		0 & \mbox{if } i \not \in A,
	\end{array}
\right.
$$
One can compose such operators $P_i$, and it is a simple observation that the order of composition does not affect the result. Hence for any subset $T \subseteq [n]$ we have an operator
$$
P_T: \Phi[(k_1, \dots, k_n)] \twoheadrightarrow \Phi[(k_h)_{h \not \in T}].
$$
We can now rewrite equation (\ref{eSmith})
as
\begin{align}
\label{Reformulate expansion equation}
(d\Phi_{(A, x)})(\sigma, \tau)=\sum_{\emptyset \neq B \subseteq [n]}\chi_{B}(\sigma)P_B(\Phi_{(A, x)})(\tau).
\end{align}
From equation (\ref{Reformulate expansion equation}), it is transparent that the restriction of $\Phi[(k_1, \dots, k_n)]$ to the commutator subgroup consists of a space of quadratic characters.\footnote{If one looks at the way we have constructed such functions, in Section \ref{expansion maps as coordinates of monomials}, then this fact is immediately clear even without invoking the more informative equation (\ref{Reformulate expansion equation}).} We actually have the following stronger statement.

\begin{proposition} 
\label{Maps restricts to characters}
The restriction to $[\widetilde{\mathcal{G}}([k_1, \dots, k_n]), \widetilde{\mathcal{G}}([k_1, \dots, k_n])]$ induces a surjective homomorphism
$$\Phi[(k_1, \dots, k_n)] \twoheadrightarrow [\widetilde{\mathcal{G}}([k_1, \dots, k_n]), \widetilde{\mathcal{G}}([k_1, \dots, k_n])]^{\vee}.
$$
Furthermore the kernel is given precisely by the span of $V_{[(k_1, \dots ,k_n)]}^{\vee}$ and the functions $\chi_{A}$ with $A \subseteq [n]$. 
\end{proposition}

\begin{proof}
The surjectivity follows immediately from Theorem \ref{governing group is universal tilde}. We next show the claim on the kernel. Recall that spaces of governing tensors with different supports are in direct sum and that, within one support $A$, the only relation is
$$
\sum_{i \in A} \sum_{x \in f(i)} \phi_{(A, x)} = 0.
$$
Therefore we deduce from Lemma \ref{lComm} that an element of the kernel must be a sum of characters and maps of the form
$$
\sum_{i \in A} \sum_{x \in f(i)} \Phi_{(A, x)}.
$$
But this last expression, thanks to Proposition \ref{Shuffling relation}, is precisely the function $\chi_{A}$. This gives the desired conclusion. 
\end{proof}

Observe that the right hand side of equation (\ref{Reformulate expansion equation}) now has an expression that is a universal function applied to $\Phi_{(A, x)}$. From this we deduce, by linearity, that for all $\Phi \in \Phi[(k_1, \dots, k_n)]$ we must have
\begin{align}
\label{Universal expansion equation}
(d\Phi)(\sigma, \tau)=\sum_{\emptyset \neq B \subseteq [n]}\chi_{B}(\sigma)P_B(\Phi)(\tau).
\end{align}
Observe that from the right hand side of equation (\ref{Universal expansion equation}), we can reconstruct each $P_{i}(\Phi)$ by plugging in an involution $g_x$ with $x \in f(i)$. Indeed such a choice will detect precisely the set $B=\{i\}$, giving in this way  $P_i(\Phi)(\tau)$. We now consider the inverse problem. Suppose that for each $i \in [n]$ we are given an element $\Phi_i \in \Phi[(k_h)_{h \neq i}]$.

\textbf{Question:} Under which conditions do we have an element $\Phi \in \Phi[(k_1, \dots, k_n)]$ such that $P_i(\Phi)=\Phi_i$ for each $i \in [n]$?  

An obvious necessary condition is that $P_j(\Phi_i)=P_i(\Phi_j)$ for every $i, j \in [n]$. It turns out to also be sufficient. We first call \emph{commuting} a vector $(\Phi_i)_{1 \leq i \leq n}$ with $\Phi_i \in \Phi[(k_h)_{h \neq i}]$ and $P_j(\Phi_i)=P_i(\Phi_j)$ for every $i, j \in [n]$. Thanks to the fact that the vector is commuting one sees that for every subset $\emptyset \neq B \subseteq [n]$ we can unambiguously define $P_B((\Phi_i)_{1 \leq i \leq n})$ by composing the operators $P_i$ in any order. With that in mind, we proceed, as a second step, to attach to any such commuting vector a $2$-cocycle
$$
\theta((\Phi_i)_{1 \leq i \leq n}):=\sum_{i \in [n]} \chi_{i}(\sigma)\Phi_i(\tau)+\sum_{\substack{B \subseteq [n] \\ \#B \geq 2}} \chi_{B}(\sigma)P_B((\Phi_i)_{1 \leq i \leq n})(\tau).
$$
This is a $2$-cocycle because each $\Phi_i$ satisfies the equation
$$
(d\Phi_i)(\sigma, \tau)=\sum_{\emptyset \neq B \subseteq [n]}\chi_{B}(\sigma)P_B(\Phi_i)(\tau),
$$
and the vector $(\Phi_i)_{1 \leq i \leq n}$ is commutative. We denote by $\text{Comm-Vect}(k_1, \dots, k_n)$ the space of all such commutative vectors. We stress that the space $\text{Comm-Vect}(k_1, \dots, k_n)$ is constructed using only the collection of spaces $\{\Phi[(k_h)_{h \neq i}]\}_{i \in [n]}$. 

For every $j \in \mathbb{Z}_{\geq 1}$ we have a natural subspace of $\Phi[(k_1, \dots, k_n)]$, which we denote $\Phi_j[(k_1, \dots, k_n)]$, consisting of functions vanishing on all nested commutators of length at least $j+1$. This induces a filtration of subspaces
$$
\{0\}=:\Phi_0[(k_1, \dots, k_n)] \subseteq \Phi_1[(k_1, \dots, k_n)] \subseteq \dots \subseteq \Phi_j[(k_1, \dots, k_n)] \subseteq \cdots,
$$
exhausting the space $\Phi[(k_1, \dots, k_n)]$, which is actually equal to the space $\Phi_n[(k_1, \dots, k_n)]$. It is a simple consequence of Proposition \ref{Maps restricts to characters} that $\Phi_1[(k_1, \dots, k_n)]$ coincides with the span of the characters and the maps $\chi_{A}$ for $A \subseteq [n]$.

We now briefly reinterpret in cohomological terms the spaces of functions obtained, together with the maps among them and the basic properties so far established.\footnote{Many thanks to Adam Morgan for pointing this out to us.} For a $[(k_1, \ldots, k_n)]$-expansion group $(G,\phi,(g_1, \ldots, g_{k_1+\ldots k_n}))$ let us denote by $N$ the subgroup $\phi^{-1}(\text{ker}(\pi_{(k_1, \ldots, k_n)}))$. Through $\pi_{(k_1, \ldots, k_n)} \circ \phi$ we identify $G/N$ and $V_{[n]}$. 

Recall that $N$ is a vector space over $\mathbb{F}_2$, thanks to the definition of a $[(k_1, \ldots, k_n)]$-expansion group. Hence $H^1(N,\mathbb{F}_2)$ is just the set of characters from $N$ to $\mathbb{F}_2$. Thanks to Shapiro's Lemma we can identify this space with $H^1(G, \mathbb{F}_2[V_{[n]}])$. Observe that to give a $1$-cocycle from $G$ to $\mathbb{F}_2[V_{[n]}]$ amounts to giving a system of $1$-cochains $\{\phi_B\}_{B \subseteq [n]}$ with $\phi_{B}:G \to \mathbb{F}_2$ satisfying the recursive equation 
$$
\phi_B(\sigma\tau) + \phi_B(\sigma) + \phi_B(\tau) = \sum_{\emptyset \neq B' \subseteq B}\chi_{B'}(\sigma)\phi_{B-B'}(\tau)
$$
for each $B \subseteq [n]$. Such an identification follows again using the monomial coordinates of Section \ref{expansion maps as coordinates of monomials}. We denote by $\text{Cocy}(G,\mathbb{F}_2[V_{[n]}])$ the space of $1$-cocycles. Thanks to equation (\ref{Univ.exp eq.}) we now see that the assignment\footnote{We extend the notation for $P_B$ by defining it to be the identity for $B= \emptyset$.}
$$
\Phi \mapsto \{P_B(\Phi)\}_{B \subseteq [n]}
$$
induces a group homomorphism 
$$
\Phi[(k_1, \dots, k_n)] \to \text{Cocy}(\widetilde{\mathcal{G}}([k_1, \dots, k_n]), \mathbb{F}_2[V_{[n]}]). 
$$
We can now easily deduce the following. 

\begin{proposition} 
\label{space of maps as 1-cocycles}
The homomorphism
$$
\Phi[(k_1, \dots, k_n)] \to \emph{Cocy}(\widetilde{\mathcal{G}}([k_1, \dots, k_n]), \mathbb{F}_2[V_{[n]}])
$$
is an isomorphism.
\end{proposition}

\begin{proof}
The map is clearly injective. We want to show that it is surjective. Both sides naturally map in $N^{\vee}$: the left hand side simply by restriction, the right hand side by mapping onto $H^1(G,\mathbb{F}_2[V_{[n]}])$ and then applying Shapiro's Lemma. One can check that the resulting square is commutative. But then thanks to Shapiro's Lemma (giving an isomorphism between the two $H^1$'s) and Proposition \ref{Maps restricts to characters}, we conclude that we only need to show that the space of coboundaries is in the image. Since the monomials $\{t_B\}_{B \subseteq [n]}$ span the group ring, we find that the space of coboundaries is spanned by the functions
$$\sigma \mapsto \sum_{B \subsetneq B' \subseteq [n]}\chi_{B'-B}(\sigma)t_{B'}.
$$
Thanks to Proposition \ref{Shuffling relation} we get that these are in the image, and we are done. 
\end{proof}

For $\Phi \in \Phi[(k_1, \dots, k_n)]$ we denote by $\text{nil-deg}(\Phi)$ the largest integer $j$ such that $\Phi \in \Phi_j[(k_1, \dots, k_n)]$.
\begin{proposition} 
\label{commuting vectors are realizable}
Let $(\Phi_i)_{1 \leq i \leq n}$ be in $\emph{Comm-Vect}(k_1, \ldots ,k_n)$. Then there exists $\Phi \in \Phi[(k_1, \dots, k_n)]$ such that $P_i(\Phi)=\Phi_i$ for every $i \in [n]$. Such $\Phi$ satisfies
$$d\Phi=\theta((\Phi_i)_{1 \leq i \leq n}).
$$
Furthermore 
$$\emph{max}(\{\emph{nil-deg}(\Phi_i)\}_{i \in [n]}) \leq \emph{nil-deg}(\Phi) \leq \emph{max}(\{\emph{nil-deg}(\Phi_i)\}_{i \in [n]})+1,
$$
with the upper bound reached in case there is at least one $i \in [n]$ with $\emph{nil-deg}(\Phi_i) \geq 2$.  
\end{proposition}

\begin{proof}
When we inflate to the free group $F_{k_1+ \dots + k_n}$, the $2$-cocycle $\theta((\Phi_i)_{1 \leq i \leq n})$ becomes a coboundary, since evidently $H^2(F_{k_1 + \dots + k_n},\mathbb{F}_2)=\{0\}$. Indeed, any central $\mathbb{F}_2$-extension of $F_{k_1 + \dots + k_n}$ is split: by definition of a free group, one can create a group theoretic section simply extending any set-theoretic sections on the set of privileged generators in the unique possible way. Therefore we obtain a system of maps $\{\phi_{B}\}_{B \subseteq [n]}$ corresponding to the group homomorphism
$$
F_{k_1+ \dots+k_n} \to \mathbb{F}_2[V_{[n]}] \rtimes V_{[n]}
$$
with $\phi_B=P_B((\Phi_i)_{i \in [n]})$ for $B \subseteq [n]$ of size less than $n-1$ and $\phi_{[n]-\{i\}}=\Phi_i$ for each $i$ in $[n]$. Now we verify that this homomorphism factors through $\widetilde{\mathcal{G}}([k_1, \ldots, k_n])$. This is a routine verification. For instance it is clear that the square of every commutator vanishes, from the shape of the $2$-cocycle $\theta((\Phi_i)_{1 \leq i \leq n})$. Similarly we see that the square of each privileged generator vanishes as does the square of the product of two generators in the same block. Therefore the collection $\{\phi_{B}\}_{B \subseteq [n]}$ lands in $\text{Cocy}(\widetilde{\mathcal{G}}([k_1, \ldots, k_n]),\mathbb{F}_2[V_{[n]}])$. Now we conclude immediately by Proposition \ref{space of maps as 1-cocycles}.

Let us now prove the last claim. 
If $\tau \in [\widetilde{\mathcal{G}}([k_1, \dots, k_n]),\widetilde{\mathcal{G}}([k_1, \dots, k_n])]$ and $i \in [n]$, then 
\begin{align}
\label{lowering equation}
([g_h,\tau])=P_i(\Phi)(\tau)
\end{align}
for every $h \in f(i)$. It follows that for every $i \in [n]$ with $\text{nil-deg}(\Phi_i) \geq 2$ we can find a nested commutator of length $\text{nil-deg}(\Phi_i)+1$ on which $\Phi$ does not vanish. Hence we see that $\text{max}(\{\text{nil-deg}(\Phi_i)\}_{i \in [n]}) \leq \text{nil-deg}(\Phi)$ in general and as soon as there is at least one $i_0 \in [n]$ with $\text{nil-deg}(\Phi_{i_0}) \geq 2$ then equation (\ref{lowering equation}) proves that 
$$
\text{nil-deg}(\Phi) \geq \text{max}(\{\text{nil-deg}(\Phi_i)\}_{i \in [n]})+1.
$$
We deduce from equation (\ref{lowering equation}) that $\Phi$ vanishes on any nested commutator of length at least $c:=\text{max}(\{\text{nil-deg}(\Phi_i)\}_{i \in [n]})+1$ and of the form $[g_h,\tau]$ for some $h$. On the other hand these span the $\mathbb{F}_2$-vector space $\widetilde{\mathcal{G}}([k_1, \dots, k_n])^{(c+1)}$ thanks to Proposition \ref{Augmentation=descending central}. Hence we obtain
$$
\text{nil-deg}(\Phi) \leq \text{max}(\{\text{nil-deg}(\Phi_i)\}_{i \in [n]})+1,
$$
which yields the full proposition.  
\end{proof}

Let now $(G, \phi,(g_1, \dots, g_{k_1+ \dots + k_n}))$ be a $[(k_1, \dots, k_n)]$-expansion group. Recall that by Theorem \ref{governing group is universal tilde}, we have that $(G, \phi,(g_1, \dots, g_{k_1+ \dots + k_n}))$ is (uniquely) a quotient of the $[(k_1, \dots, k_n)]$-expansion group $\widetilde{\mathcal{G}}([k_1, \dots, k_n])$. This gives us a subspace of $\Phi[(k_1, \dots, k_n)]$ consisting of the functions that are well defined modulo the projection onto $G$. We denote it as $\Phi(G, \phi,(g_1, \dots, g_{k_1+ \dots + k_n}))$. Observe that this space uniquely determines the $[(k_1, \dots, k_n)]$-expansion group $(G, \phi,(g_1, \dots, g_{k_1+ \dots + k_n}))$. Indeed the restriction map 
$$
\Phi[(k_1, \dots, k_n)] \twoheadrightarrow [\widetilde{\mathcal{G}}([k_1, \dots, k_n]), \widetilde{\mathcal{G}}([k_1, \dots, k_n])]^{\vee}
$$
is surjective. Then the group kernel of the projection from $\widetilde{\mathcal{G}}([k_1, \dots, k_n])$ to $G$, which must be a linear subspace of $[\widetilde{\mathcal{G}}([k_1, \dots, k_n]), \widetilde{\mathcal{G}}([k_1, \dots, k_n])]$,  is equal to the maximal normal subgroup $N$ of $\widetilde{\mathcal{G}}([k_1, \dots, k_n])$ such that all elements of $\Phi(G, \phi,(g_1, \dots, g_{k_1+ \dots + k_n}))$ factor through $N$. So determining $(G, \phi,(g_1, \dots, g_{k_1+ \dots + k_n}))$ amounts to determining the space $\Phi(G, \phi,(g_1, \dots, g_{k_1+ \dots + k_n}))$. Now consider the space of commuting vectors $(\Phi_i)_{1 \leq i \leq n}$ such that $\Phi_i \in \Phi(G, \phi,(g_1, \dots, g_{k_1+ \dots + k_n})_i)$. We denote it with
$$
\text{Comm-Vect}(\{(G, \phi,(g_1, \dots, g_{k_1+ \dots + k_n}))_i\}_{i \in [n]}).
$$
We stress that this space is constructed merely invoking the set of $n$ corners of the $[(k_1, \dots, k_n)]$-expansion group $(G, \phi,(g_1, \dots, g_{k_1+ \dots + k_n}))$. We denote by
$$
2-\text{Cocy}(\{\Phi(G, \phi,(g_1, \dots, g_{k_1+ \dots + k_n}))_i\}_{i \in [n]})
$$ 
the corresponding set of $2$-cocycles constructed as explained above, but this time only using elements of $\text{Comm-Vect}(\{(G, \phi,(g_1, \dots, g_{k_1+ \dots + k_n}))_i\}_{i \in [n]})$. Again, we stress that this space of $2$-cocycles is constructed using only the set of $n$-corners of the $[(k_1, \dots, k_n)]$-expansion group $(G, \phi,(g_1, \dots, g_{k_1+ \dots + k_n}))$.

Intersecting $\Phi(G, \phi,(g_1, \dots, g_{k_1+ \dots + k_n}))$ with the nilpotency filtration induces a filtration of spaces 
$$\{0\}=\Phi_0(G, \phi,(g_1, \dots, g_{k_1+ \dots + k_n})) \subseteq \dots \subseteq \Phi_j(G, \phi,(g_1, \dots, g_{k_1+ \dots + k_n})) \subseteq \dots
$$
exhausting the space $\Phi(G, \phi,(g_1, \dots, g_{k_1+ \dots + k_n}))$, which is actually equal to 
$$
\Phi_{c(G)}(G, \phi,(g_1, \dots, g_{k_1+ \dots + k_n})),
$$ 
where $c(G)$ is the nilpotency class of $G$. We denote by 
$$\text{Comm-Vect}_j(\{(G, \phi,(g_1, \dots, g_{k_1+ \dots + k_n}))_i\}_{i \in [n]})
$$ the subspace of $\text{Comm-Vect}(\{(G, \phi,(g_1, \dots, g_{k_1+ \dots + k_n}))_i\}_{i \in [n]})$, where all the coordinates have $\text{nil-deg}$ at most $j-1$. 
 
Now combining Proposition \ref{Maps restricts to characters}, equation (\ref{Universal expansion equation}) and Proposition \ref{commuting vectors are realizable}, we can answer the question at the beginning of this section. Namely if $j \in \mathbb{Z}_{\geq 2}$, to obtain the space 
$$
\Phi_j(G, \phi,(g_1, \dots, g_{k_1+ \dots + k_n}))
$$ 
one needs to do four steps.\footnote{The case $j=1$, as remarked above, equals the span of the characters and the maps $\chi_{A}$ so it is given.} \\
$\textbf{Step 1}:$ Obtain the spaces $\Phi_{j-1}(G, \phi,(g_1, \dots, g_{k_1+ \dots + k_n}))_i$ for every $i \in [n]$. \\
$\textbf{Step 2}:$ Check which of the $2$-cocycles obtained from 
$$\text{Comm-Vect}_j(\{(G, \phi,(g_1, \dots, g_{k_1+ \dots + k_n}))_i\}_{i \in [n]})
$$ is trivial in $H^2(G,\mathbb{F}_2)$. \\
$\textbf{Step 3}:$ For each of the elements obtained in Step $2$, any corresponding $1$-cochain factors through $G$ and hence yields an element of $\Phi_j(G, \phi,(g_1, \dots, g_{k_1+ \dots + k_n}))$. \\
$\textbf{Step 4}:$ The collection of all the elements of $\Phi_j(G, \phi,(g_1, \dots, g_{k_1+ \dots + k_n}))$ obtained in Step $3$ span the space $\Phi_j(G, \phi,(g_1, \dots, g_{k_1+ \dots + k_n}))$.

Note that it is not at all obvious how to perform \textbf{Step 2}, since the unknown group $G$ occurs in this step. Fortunately, there will be a simple universal criterion to perform $\textbf{Step 2}$ in our arithmetical application. Namely, the cohomology classes will vanish if and only if their inflation to $G_{\mathbb{Q}}$ does. This fact relies on the special form of the $2$-cocycles appearing on the right hand side of equation (\ref{Universal expansion equation}). Therefore the above abstract procedure will boil down to an effective iterative procedure to construct the $j$-th higher genus space of a multi-quadratic field $K$ out of the $(j-1)$-th higher genus spaces of the fields $\{K_{[n]-\{i\}}\}_{i \in [n]}$. 

\section{Proofs of main theorems} 
\label{Expansion arithmetic}
\subsection{Proof of Theorems \ref{theoremA}, \ref{controlling gn} and \ref{Presenting the group}}
Let $(a_1, \dots, a_n)$ be an acceptable vector and denote $k_i:=\omega(a_i)$. Recall that if $i \in [n]$ and $p$ is a prime dividing $a_i$, then $p$ has ramification degree equal to $2$ in $H_2^{+}(a_1, \dots, a_n)/\mathbb{Q}$. Therefore any inertia subgroups at $p$ is a subgroup of size equal to $2$. So each of them has precisely one non-trivial element, and by a choice of inertia at $p$ we mean the choice of such an involution in $\text{Gal}(H_2^{+}(a_1, \dots, a_n)/\mathbb{Q})$. Now write 
$$
a_i:=p_{h_i + 1} \cdot \ldots \cdot p_{h_i+ k_i},
$$
where $h_i:=\sum_{1 \leq j \leq i - 1} k_j$. In this way there is a bijection between $[k_1+ \dots + k_n]$ and the prime factors of $a_1 \cdot \ldots \cdot a_n$. The following clarifies the relevance of $[(k_1, \dots, k_n)]$-expansion groups when looking at $\text{Gal}(H_2^{+}(a_1, \dots, a_n)/\mathbb{Q})$.

\begin{proposition} 
\label{Galois groups of H2(a1,.., An) is expansion}
Choose for every $j \in [k_1+ \dots + k_n]$ an inertia element $\sigma_j$ at $p_j$. Then 
$$(\emph{Gal}(H_2^{+}(a_1, \dots, a_n)/\mathbb{Q}),(\chi_{p_1}, \dots, \chi_{p_{k_1+ \dots + k_n}}), (\sigma_1, \dots, \sigma_{k_1+ \dots + k_n})),
$$
is a $[(k_1, \dots, k_n)]$-expansion group. 
\end{proposition}

\begin{proof}
As we explained at the beginning of this section, the inertia elements are involutions. Furthermore, they form a dual basis of $\{\chi_{p_1}, \dots, \chi_{p_{k_1+ \dots + k_n}}\}$, since the $a_i$ are pairwise coprime. The largest abelian subextension of $\mathbb{Q}$ contained in $H_2^{+}(a_1, \dots, a_n)/\mathbb{Q}$ is the multi-quadratic number field obtained by adjoining all the square roots of the prime factors of the various $a_i$, thanks to the fact that each of them is $1$ modulo $4$. 

This means that the commutator subgroup is the kernel of $(\chi_{p_1}, \dots, \chi_{p_{k_1+ \dots + k_n}})$. Finally, by definition, the group $\text{Gal}(H_2^{+}(a_1, \dots, a_n)/\mathbb{Q}(\sqrt{a_1}, \dots, \sqrt{a_n}))$ is a $\mathbb{F}_2$-vector space. This concludes the proof. 
\end{proof}

Now Theorem \ref{Presenting the group} follows immediately upon combining Theorem \ref{governing group is universal tilde} and Proposition \ref{Galois groups of H2(a1,.., An) is expansion}. Hence Theorem \ref{theoremA} falls as a consequence of Theorem \ref{Presenting the group}. 

To prove Theorem \ref{controlling gn}, we recall that a character $\chi \in \text{Cl}^{+}(\Q(\sqrt{a_1}, \dots, \sqrt{a_n}))^{\vee}[2]$ belongs to $\text{Gn}(K, j)$ if it vanishes on all $(j + 1)$-th nested commutators with entries in $G_\Q$. We can turn $\text{Gal}(H_2^{+}(a_1, \dots, a_n)/\mathbb{Q})$ in an $n$-expansion group $G$ by Proposition \ref{Galois groups of H2(a1,.., An) is expansion} and then attach to it its Lie algebra $L_\bullet$. The case $j = 1$ is classical, so suppose that $j > 1$. Then there is a natural map $\text{Gn}(K, j) \rightarrow L_j^\vee$ by restricting $\chi$ to $G^{(j)}$, which is well-defined since $\chi$ vanishes on $G^{(j + 1)}$ and $j > 1$. Since $\text{Gn}(K, j - 1)$ is contained in the kernel of this map, Theorem \ref{controlling gn} then follows from Proposition \ref{counting more}, Proposition \ref{Consistent tensors are governing tensors more general} and Proposition \ref{expansion bracket give tilde consistent tensors}.

\subsection{Proof of Theorem \ref{theoremC}}
We say that $(a_1, \dots, a_n)$ is strongly quadratically consistent if every prime divisor of $a_i$ is a square modulo every prime divisor of $a_j$ for distinct $i, j \in [n]$. 

\begin{proposition} 
\label{quadratically consistent}
Suppose that $(a_1, \dots, a_n)$ is maximal. Then $(a_1, \dots, a_n)$ is strongly quadratically consistent.
\end{proposition}
 
\begin{proof}
Indeed, we must have an expansion map with support $\{\chi_{a_i}, \chi_q\}$ and pointer $\chi_q$ for every prime $q$ dividing $a_j$, sine $(a_1, \dots, a_n)$ is maximal. The existence of such an expansion map is equivalent to $\chi_{a_i} \cup \chi_q$ being trivial in $H^2(G_{\mathbb{Q}}, \mathbb{F}_2)$. Going locally at a prime divisor $p$ of $a_i$, we see that $\chi_q$ must then become locally trivial, i.e. $q$ is a square modulo $p$. 
\end{proof}

We can now complete the proof of Theorem \ref{theoremC}. Thanks to Theorem \ref{governing group is universal tilde}, we know that $(a_1, \dots, a_n)$ is maximal if and only if for every $i_0 \in [n]$, every $b$ dividing $a_{i_0}$ and every $A \subseteq [n]-\{i_0\}$, we can find an expansion map for $G_{\mathbb{Q}}$ with support $\{\chi_{a_i}\}_{i \in A} \cup \{\chi_b \}$ and pointer $\chi_{b}$ inducing an unramified extension of $\mathbb{Q}(\{\sqrt{a_i}\}_{i \in A}, \sqrt{b})$. From this it is transparent that if $(a_1, \dots, a_n)$ is maximal then, for each $j$ in $[n]$, the vector $(a_h)_{h \neq j}$ is maximal. 

Furthermore, thanks to Proposition \ref{quadratically consistent} we see that for every $j \in A$ and each prime divisor $p$ of $a_j$ the right hand side of equation (\ref{eSmith}) is trivial in $H^2(G_{\mathbb{Q}_p},\mathbb{F}_2)$ if and only if $p$ splits completely in the field corresponding to the expansion map $\phi_{A-\{j\}}$ with pointer $\chi_b$. Since this holds for all $A$ and $b$, we conclude that $p$ splits completely in $H_2^{+}((a_h)_{h \neq j})$. Hence we have shown that if $(a_1, \dots, a_n)$ is maximal, then for every $j \in [n]$, the vector $(a_h)_{h \neq j}$ is maximal and each prime divisor of $a_j$ splits completely in $H_2^{+}((a_h)_{h \neq j})$. 

To obtain the converse one inverts the logic above and obtains that if for every $j \in [n]$ it is given that $(a_h)_{h \neq j}$ is maximal and each prime divisor of $a_j$ splits completely in $H_2^{+}((a_h)_{h \neq j})$, then the $2$-cocycles appearing on the right hand side of equation (\ref{eSmith}) are all everywhere locally trivial and hence trivial in $H^2(G_{\mathbb{Q}}, \mathbb{F}_2)$. Hence, to conclude, in virtue of Theorem \ref{governing group is universal tilde} one needs to show that each of the corresponding twist families of $\mathbb{F}_2$-central extensions admits an unramified representative. This can be obtained with a straightforward adaptation to $l=2$ of \cite[Proposition $4.10$]{koymans--pagano}.

\subsection{\texorpdfstring{Obtaining $\text{Gn}(K_{[n]}, j)$ from $(\text{Gn}(K_{[n]-\{i\}}, j - 1))_{i \in [n]}$}{Obtaining the next genus space}} 
\label{procedure: arithmetic}
We begin with the elementary observation that for all acceptable vectors $(a_1, \dots, a_n) \in \mathbb{Z}_{\geq 1}^{n}$, all $i \in [n]$, and for any choice of inertia elements in $\text{Gal}(H_2^{+}(a_1, \dots, a_n)/\mathbb{Q})$, we have that the $i$-th corner of $\text{Gal}(H_2^{+}(a_1, \dots, a_n)/\mathbb{Q})$ is $\text{Gal}(H_2^{+}((a_h)_{h \in [n]: h \neq i})/\mathbb{Q})$, where we view this last Galois group as an $[(\omega(a_h))_{h \in [n]: h \neq i}]$-expansion group with the choice of inertia elements induced by the choices made for $\text{Gal}(H_2^{+}(a_1, \dots, a_n)/\mathbb{Q})$.

Hence to show Theorem \ref{Last thm}, all that we need to prove is that $\textbf{Step 2}$ in the procedure at the end of Section \ref{Reconstructing from corners} can be replaced by the definition of $\text{Comm-Vect}_j^{\circ}(a_1, \dots, a_n)$ given in the introduction. In other words we have to show that as soon as we have a commutative vector $\underline{v} \in \text{Comm-Vect}_j^{\circ}(a_1, \dots, a_n)$, then there is a continuous $1$-cochain $\Phi:G_{\mathbb{Q}} \to \mathbb{F}_2$ such that
$$
(d\Phi)=\theta(\underline{v}),
$$
and the smallest Galois extension of $\Q$ through which $\Phi$ factors, called \emph{the field of definition} of $\Phi$ and denoted $L(\Phi)$, yields an \emph{unramified} central $\mathbb{F}_2$-extension of the field $L(\underline{v})$, where $L(\underline{v})$ denotes the compositum of the fields of definitions of the coordinates of $\underline{v}$. By definition, we certainly have a continuous $1$-cochain $\Phi_0:G_{\mathbb{Q}} \to \mathbb{F}_2$ such that
$$
(d\Phi_0)=\theta(\underline{v}).
$$
So we have to find a character $\chi:G_{\mathbb{Q}} \to \mathbb{F}_2$ such that $\Phi:=\Phi_0+\chi$ has the ramification condition just explained. Following the same argument provided in \cite[Proposition $4.10$]{koymans--pagano}, it is enough to check that the cocycle $\theta(\underline{v})$ splits locally at inertia for every prime $p$ dividing $a_1 \cdot \ldots \cdot a_n$. Pick such a prime $p$, and let $i$ be the corresponding element of $[n]$ such that $p$ divides $a_i$ (recall that $a_1, \dots, a_n$ are pairwise coprime). Pick an inertia element $\sigma_p$ in $\text{Gal}(L(\underline{v})/\mathbb{Q})$. 

We claim that in the sum of the $2^n - 1$ terms defining $\theta(\underline{v})(\sigma, \tau)$, each of these terms will vanish when we plug in $\sigma:=\sigma_p$ and $\tau:=\sigma_p$. Indeed, if a non-empty subset $B \subseteq [n]$ does not contain $i$, then certainly $\chi_B(\sigma_p)=0$. Assume now that $i \in B$. Then the $P_{B-\{i\}}(v_i)$ has field of definition contained in $L(v_i)$, where $v_i$ is the $i$-th coordinate of $\underline{v}$. These maps are in particular in $\Phi_{j-1}((a_h)_{h \neq i})$, so the field $L(v_i)/\mathbb{Q}$  is unramified at $p$. It follows that the image of $\sigma_p$ in $\text{Gal}(L(v_i)/\mathbb{Q})$ equals the trivial element $\text{id}$. But now, by construction, the map $P_{B-\{i\}}(v_i)$ vanishes on the identity element. Hence our claim follows. 

Therefore it follows that in the central extension defined by $\theta(\underline{v})$ the inertia element $\sigma_p$ must necessarily lift to an involution. Hence the extension does not ramify, since any potential ramification (recalling that $p$ is odd) would be necessarily tame and thus with cyclic inertia, and therefore it should yield that $\sigma_p$ lifts to an element of order $4$, which we have just disproved.\footnote{Observe that in a $\mathbb{F}_2$-central extension of $G$, given by a $2$-cocycle $\theta$ with $\theta(\text{id}, \text{id})=0$, an involution $g$ of $G$ is either lifted to an involution or to an element of order $4$ depending respectively on whether $\theta(g,g)=0$ or $\theta(g,g) \neq 0$.} This ends the proof of Theorem \ref{Last thm}.

\begin{remark}
We say that an acceptable vector $(a_1, \dots, a_n)$ is quadratically consistent in case for each distinct $h, k \in [n]$ and every prime factor $p$ of $a_h$ one has that $a_k$ is a square modulo $p$. We remark that in case $(a_1, \dots, a_n)$ is quadratically consistent then, by the same calculation done in the proof of Theorem \ref{theoremC}, one has a rather convenient criterion to decide whether an element $(\Phi_1, \dots, \Phi_n) \in \text{Comm-Vect}_j(a_1, \dots, a_n)$ actually belongs to $\text{Comm-Vect}_j^{\circ}(a_1, \dots, a_n)$. Namely that happens if and only if every prime factor $p$ of $a_h$ splits completely in $L(\Phi_h)/\mathbb{Q}$ for every $h \in [n]$.
\end{remark}

\end{document}